\documentclass[preprint,12pt]{elsarticle}

\usepackage{amssymb}

\usepackage[utf8]{inputenc}
\usepackage{amsmath}
\usepackage{amsfonts}
\usepackage{amsthm}
\usepackage{graphicx}
\usepackage[colorinlistoftodos]{todonotes}
\usepackage{soul}
\usepackage{tikz}
\usepackage{caption}
\usepackage{subcaption}
\usepackage[colorlinks=true, allcolors=blue]{hyperref}

\usepackage{graphics}
\usepackage{array}

\newtheorem{theorem}{Theorem}[section]
\newtheorem{corollary}[theorem]{Corollary}
\newtheorem{lemma}[theorem]{Lemma}

\newtheorem{observation}[theorem]{Observation}

\theoremstyle{definition}
\newtheorem{definition}[theorem]{Definition}

\newcounter{claimcounter}
\newcounter{claimproofcounter}

\newcommand{\cD}{\mathcal{D}}

\journal{Advances in Applied Mathematics}

\begin{document}

\begin{frontmatter}



\title{Reconstructibility of unrooted level-$k$ phylogenetic networks from distances}


\author[Delft]{Leo van Iersel}
\ead{L.J.J.vanIersel@tudelft.nl}

\author[EA]{Vincent Moulton}
\ead{V.Moulton@uea.ac.uk}

\author[Delft]{Yukihiro Murakami\corref{cor1}}
\ead{Y.Murakami@tudelft.nl}

\address[Delft]{Delft Institute of Applied Mathematics, Delft University of Technology, Van Mourik Broekmanweg 6, 2628 XE, Delft, The Netherlands}
\address[EA]{School of Computing Sciences, University of East Anglia, NR4 7TJ, Norwich, United Kingdom}

\cortext[cor1]{Corresponding author}


\begin{abstract}
    A phylogenetic network is a graph-theoretical tool
    that is used by biologists to represent the evolutionary history of a collection of species.
    One potential way of constructing such networks is via a distance-based approach, where one is asked to find a phylogenetic network that in some way represents a given distance matrix, which gives information on the evolutionary distances between present-day taxa.
    Here, we consider the following question.
    For which~$k$ are unrooted level-$k$ networks uniquely determined by their distance matrices?
    We consider this question for shortest distances as well as for the case that the multisets of all distances is given.
    We prove that level-$1$ networks and level-$2$ networks are reconstructible from their shortest distances and multisets of distances, respectively.
    Furthermore we show that, in general, networks of level higher than~$1$ are not reconstructible from shortest distances and that networks of level higher than~$2$ are not reconstructible from their multisets of distances.
\end{abstract}




\end{frontmatter}


\section{Introduction}

Phylogenetic trees are often used to represent the evolutionary history of species, or more generally, taxa~\cite{felsenstein2004inferring}.
Trees can be a powerful tool  
for elucidating relationships between species, especially in case the 
species in question have evolved only via speciation events.
However, other events often also drive evolution, including 
hybridization, introgression, and lateral gene transfer.
When such reticulate events occur, more general graphical structures,
known as phylogenetic networks~\cite{bapteste2013networks,huson2010phylogenetic} 
can be a useful addition to trees.

There are two main types of phylogenetic networks: rooted and unrooted networks.
A rooted network is a directed acyclic graph that represents how extant taxa have evolved from a single common ancestor, also known as the root.
Internal vertices denote either speciation or reticulate events, and edges have directions to indicate the transfer of genetic material between the two vertices that are incident to it.
Unrooted networks have similar properties except they have no direction on the edges.
A lack of direction could, for example, represent an ambiguity in knowledge of the
direction in which genetic material is transferred  between species.
Note that every rooted network has an underlying unrooted network, that can be obtained by suppressing the root vertex and ignoring edge directions.
Conversely, one can try to obtain a rooted network from an unrooted network by estimating the location of the root via an outgroup, if it is known which vertices represent reticulations~\cite{huber2019rooting}.
In this paper we will only consider unrooted networks, which we shall call networks for short.
We present an example of such a network in Figure~\ref{fig:Example}.

As the shift from phylogenetic 
trees to networks has become more prevalent in the biological 
literature, finding good ways to construct phylogenetic networks 
has become a core theme in phylogenetics.
Such an undertaking has experienced major developments through various reconstruction approaches (e.g., maximum-likelihood~\cite{jin2006maximum}; building blocks~\cite{van2018polynomial,Murakami2019,van2019practical}; distance-based~\cite{bryant2004neighbor, bordewich2018recovering}; see~\cite{huson2010phylogenetic} for an overview).
In this paper we consider the distance-based approach, in which one is given a distance matrix on the set of taxa in question and then aims to build a network representing this matrix.
An entry in a distance matrix gives the evolutionary distance, a measure of genetic divergence between distinct taxa.
This raises the following question.
`Is there a network that precisely represents the given distance matrix?'

The groundwork for distance-based methods is well-established for phylogenetic trees~\cite{sokal1958statistical,saitou1987neighbor,felsenstein2004inferring,pardi2016distance}.
For networks, the story is more complicated.
Since networks can contain cycles, there can be more than one path between two taxa, which can lead to more than one distance.
This results in various types of distances that can be associated to a network.
Two such types, which we cover in this paper, include the shortest distances and the multisets of distances.
For the shortest distances, we search for a network in which the distance of a shortest path between each pair of taxa coincides with the matrix; for the multisets of distances, we search for a network in which the multiset of distances of all paths between each pair of taxa coincides with the matrix.
In Figure~\ref{fig:Example}, we present a network with its multisets of distances.
The shortest distance matrix can be worked out from the multisets of distances by taking the smallest element for each matrix entry.

Before proceeding any further, we must acquaint ourselves with two similar, yet subtly different notions that are vital in understanding distance-based methods for networks.
One can either \emph{construct} or \emph{reconstruct} networks from distance matrices.
Constructing a network means that we initially start with a distance matrix and come up with a network that is consistent in some way with such a matrix.
Some classical network construction methods from distances include Neighbor-Net~\cite{bryant2004neighbor} and T-Rex~\cite{makarenkov2001t}.
In the process, one is sometimes interested in finding a network that optimizes some particular criterion, such as the hybridization number~\cite{chan2006reconstructing, chang2018reconstructing}.
Networks obtained via construction methods are often non-unique, which is the biggest distinction between construction 
and reconstruction methods.

Reconstructing a network means that we start with a network, find the distance matrix that is associated to it (e.g., shortest distances), and try to reconstruct the original network from the distance matrix.
The goal then is to decide which networks can be uniquely 
reconstructed from their distances, in other words, to decide upon the \emph{reconstructibility} of 
different classes of networks from their various distance matrices.
The main results of~\cite{bryant2007consistency,bordewich2016determining,bordewich2016algorithm,bordewich2018constructing, bordewich2018recovering, willson2006unique,hayamizu2020recognizing} follow this exact format; they 
show that some unrooted/rooted networks (or a representative of the equivalence class) 
can be reconstructed from certain distance matrices.
Roughly speaking, they show that within a particular network class, if two networks have the same particular distance matrix then the networks are equivalent.
Interestingly, although distance-based reconstruction results have been recently 
developed for rooted networks, similar problems have been less 
studied for unrooted networks. 

As a first step in this direction, we focus on reconstructing unweighted unrooted networks.
Every edge in the network has a weight of~$1$, which means that distances between two taxa correspond to the number of edges contained in paths between two taxa.
Now, to identify which networks are reconstructible from certain distance matrices, we call on the notion of the level of a network.
The \emph{level} of a network is the maximum number of edges that need to be deleted from a biconnected component to obtain a tree~\cite{gambette2012quartets}.
In this paper we consider the problem of reconstructing level-$k$ networks in general, both from their shortest distances and their multisets of distances.

A recent paper has shown that optimal cactus graphs are reconstructible from their shortest distances, while in general there could be many cactus graphs that realize the same shortest distances~\cite{hayamizu2020recognizing}.
Cactus graphs are connected graphs in which each edge belongs to at most one cycle -- these graphs are a generalization of level-$1$ networks.
Here, an optimal network refers to one that realizes the shortest distance matrix, in which the total sum of edge weights is minimal.
The difference between this result and our result is that we consider unweighted networks, for which we may leave out the optimality restriction.
The problem of reconstructing cactus graphs has also been of interest
within the graph theory literature.
Some have considered reconstructing them from subgraphs~\cite{geller1969reconstruction}, and others from shortest path information~\cite{kranakis2016reconstructing}, which are both different from the distance data that we consider in this paper.
Therefore, our problem of reconstructing networks from distances is fundamentally different from both of these papers.

The rest of the paper is organised as follows.
In the next section we introduce basic definitions and notations.
In Section~\ref{sec:Negative}, we show that in general, level-$2$ networks are not reconstructible from their shortest distances (Lemma~\ref{lem:L2NSD}), and that networks of level higher than~$2$ are not reconstructible from their shortest distances nor from their multisets of distances (Lemma~\ref{lem:LkNSM}).
In Section~\ref{sec:ShortestDistance}, we show that level-$1$ networks as well as level-$2$ networks on fewer than~$4$ leaves are reconstructible from their shortest distances (Theorem~\ref{thm:L1ReconstructibleShortPath} and Lemma~\ref{lem:L2<3ShortestRecon}).
In Section~\ref{sec:Multiset}, we show that level-$2$ networks are reconstructible from their multisets of distances (Theorem~\ref{thm:L2ReconstructibleMultiset}).
We conclude with a discussion in Section~\ref{sec:Discussion} on open problems and possible future directions in this area.

\section{Preliminaries}\label{sec:preliminaries}

\begin{definition}
    Let~$X$ be a non-empty finite set.
    An \emph{(unweighted unrooted binary phylogenetic) network}~$N$ on~$X$ is a simple graph (an unweighted, undirected graph with no loops or multiple edges) with
    \begin{enumerate}
        \item $|X|$ vertices of degree-1 (the \emph{leaves}); and
        \item all other vertices of degree-3 (the \emph{internal vertices}).
    \end{enumerate}
    The leaves are bijectively labelled by the set~$X$.
    If~$|X| = 1$ then we define the singleton graph with one vertex labelled by the element of~$X$ as the network on~$X$.
    A network with no cycles is a \emph{(phylogenetic) tree}.
\end{definition}

\emph{Deleting an edge}~$uv$ from a network is the action of removing the edge~$uv$ and suppressing any degree-$2$ vertices in the resulting subgraph.
\emph{Deleting a vertex} from a network is the action of removing the vertex, deleting all its incident edges, and suppressing any degree-$2$ vertices in the resulting subgraph.
A \emph{cut-edge} of a network is an edge whose deletion disconnects the network.
We call a cut-edge \emph{trivial} if the edge is incident to a leaf, and \emph{non-trivial} otherwise.
Note that for a network~$N$ on~$X$, deleting a cut-edge breaks the network into two components.
The leaf-set~$X$ can be partitioned into the leaves that are contained in one component and the leaves that are contained in the other; therefore every cut-edge of a network \emph{induces} a partition~$X=Y\cup Z$ of~$X$ (where one of~$Y$ or~$Z$ could possibly be empty).
These partitions are not unique in general (i.e., two distinct cut-edges can induce the same partition).
Upon cutting a non-trivial cut-edge, if one of the components is a tree, then we say that the subgraph that corresponds to this component is a \emph{pendant subtree}. 
Given a cut-edge~$uv$ we say that a leaf~$x$ can be \emph{reached from~$u$ via~$uv$} if, upon deleting the edge~$uv$ without suppressing degree-$2$ vertices,~$x$ is in the same component as~$v$ in the resulting subgraph.

A \emph{biconnected component (blob)} of a network is a maximal 2-connected subgraph with at least three vertices.
We say that a network is a \emph{level-$k$ network} if at most~$k$ edges must be deleted from every blob to obtain a tree.
We say that a leaf is \emph{contained} in a blob if the neighbour of the leaf is a vertex of the blob.
A cut-edge is \emph{incident} to a blob if one of the endpoints of the edge is a vertex of the blob.
A blob is \emph{pendant} if there is exactly one non-trivial cut-edge that is incident to the blob.
We say that a leaf~$x$ can be \emph{reached} from a blob~$B$ via a cut-edge~$uv$ if~$u$ is a vertex of~$B$ and~$x$ can be reached from~$u$ via~$uv$.

Let~$N$ be a network on~$X$ and let~$x$ and~$y$ be leaves in~$N$.
We recall the notation used in~\cite{bordewich2018recovering}.
The \emph{multiset of distances} between~$x$ and~$y$, denoted~$d(x,y)$ (and sometimes as~$d^N(x,y)$ where necessary), is the multiset consisting of lengths of all possible paths between~$x$ and~$y$ in~$N$.
Since~$N$ is an unweighted network, the length of a path is simply the number of edges contained in the path.
We let~$\mathcal{D}(N)$ denote the~$|X|\times |X|$ matrix whose~$(x,y)$-th entry is~$d(x,y)$.
We further define the \emph{shortest distance} between~$x$ and~$y$, denoted~$d_m(x,y)$, by taking~$d_m(x,y) = \min d(x,y)$.
We analogously define~$\mathcal{D}_{m}(N)$ to be the~$|X| \times |X|$ matrix whose~$(x,y)$-th entry is~$d_m(x,y)$.
An example of a network with its multisets of distances is illustrated in Figure~\ref{fig:Example}.

We use the following notation for the multisets.
A \emph{multiset} is a tuple~$(A,m)$ where~$A$ is a set and~$m$ is a function that specifies the multiplicity of each element in~$A$.
For~$x\notin A$, we let~$m(x) = 0$.
We will, for the most part, write~$(A,m)$ as~$A = \{a_1^{m(a_1)},\ldots, a_k^{m(a_k)}\}$.
Let~$n$ be an integer.
We let~$A-n$ denote the multiset obtained by subtracting~$n$ from each element of~$A$ (i.e.,~$A-n = \{(a_i-n)^{m(a_i)}:i\in[k]\}$.)
Given two multisets~$(A,m_A)$ and~$(B,m_B)$, the \emph{sum}~$A+B$ is defined as the multiset~$(A\cup B, m_{A+B})$ where~$m_{A+B}= m_A(x) + m_B(x)$ for~$x\in A\cup B$.

A network~$N$ \emph{realises} the multisets of distances~$\cD$ if~$\cD(N) = \cD$.
Similarly, a network~$N$ realises the shortest distances~$\cD_m$ if~$\cD_m(N) = \cD_m$.
As we will show in the next section, there could be many distinct networks that realise the same distance matrix.
Therefore we emphasize the following notion.

\begin{definition}
    A network~$N$ is \emph{reconstructible} from its multisets of distances (respectively the shortest distances) if~$N$ is the only network that realises~$\cD(N)$ (respectively $\cD_m(N))$.
\end{definition}

We now introduce two substructures of networks, the \emph{cherry} and the \emph{chain}, which are key ingredients in proving the main results of this paper.
\begin{definition}
    Two leaves~$x$ and~$y$ form a \emph{cherry} if they share a common neighbour.
\end{definition}
Observe that~$x$ and~$y$ form a cherry if and only if~$d(x,y) = \{2\}$.
In addition,~$x$ and~$y$ form a cherry if and only if~$d_m(x,y) = 2$.
\begin{definition}
    A \emph{chain} of length~$k\geq1$ is a~$k$-tuple of leaves~$(a_1,\ldots,a_k)$ such that~$d_m(a_i,a_{i+1}) = 3$ for all~$i\in [k-1] = \{1,\ldots, k-1\}$.
\end{definition}

Call a chain~$(a_1,\ldots, a_k)$ \emph{maximal} if there is no chain~$(b_1,\ldots, b_\ell)$ such that~$\{a_1,\ldots,a_k\}\subsetneq \{b_1,\ldots,b_\ell\}$. 
We assume all chains to be maximal, unless stated otherwise.
Two chains~$(a_1,\ldots,a_k)$ and~$(b_1,\ldots,b_\ell)$ are \emph{adjacent} if~$d_m(a_i,b_j)=4$ for at least one of~$i\in\{1,k\}$ and~$j\in\{1,\ell\}$.
Two chains are \emph{adjacent twice} if~$d_m(a_1,b_1) = d_m(a_k,b_\ell) = 4$ or if~$d_m(a_1,b_\ell) = d_m(a_k,b_1) = 4$.

Given a chain~$a = (a_1,\ldots,a_k)$, let~$p_i$ denote the neighbour of the leaf~$a_i$ for~$i\in[k]$.
The edges~$p_ip_{i+1}$ for~$i\in[k-1]$ are called the \emph{edges of the chain}.
We say that the chain is \emph{incident to cut-edges} if the edges of the chain are cut-edges.
Observe that one of these edges is a cut-edge if and only if they are all cut-edges.
We say that the chain is \emph{contained in a blob~$B$} if the edges of the chain are edges in~$B$.
Observe that one of these edges is an edge of~$B$ if and only if they are all edges in~$B$.

Note that a leaf can be in both a cherry and a chain.
In a network without cherries, it is possible to partition the leaves into chains. 

Let~$B$ be a level-$2$ blob of some network~$N$. 
We may obtain the \emph{generator} of~$B$ by deleting all cut-edges that are incident to~$B$ and taking the component that is~$B$.
The edges of the generator of~$B$ are called the \emph{sides} of the generator, or simply the sides of~$B$.
Let~$N$ be a network with no pendant subtrees, let~$e$ be a side of~$B$, and let~$x$ be a leaf in~$N$.
If the neighbour of~$x$, say~$p$, subdivides~$e$ in~$N$ then we say that \emph{$x$ is on the side~$e$} or that \emph{the side~$e$ contains~$x$}.
We say that a chain~$a = (a_1,\ldots,a_k)$ \emph{is on the side~$e$} or that \emph{the side~$e$ contains the chain~$a$} if every leaf~$a_i$ in the chain is on the side~$e$.
If an endpoint of a cut-edge~$uv$ subdivides~$e$ then we say that \emph{the side~$e$ is incident to~$uv$.}

For an overview of the definitions presented in this section, see Figure~\ref{fig:Example}.

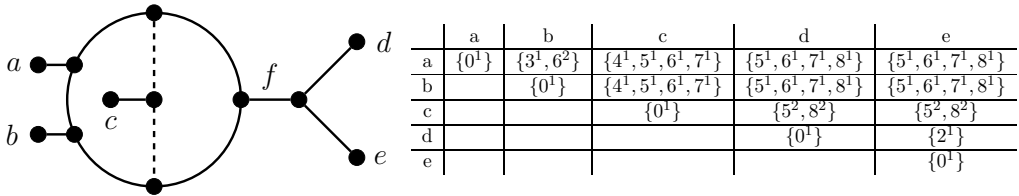
\begin{figure}[h]
\centering
\begin{minipage}[c]{.40\textwidth}
 \resizebox{\textwidth}{!}{
 \begin{tikzpicture}[every node/.style={draw, circle, fill, inner sep=0pt}]
 	\tikzset{edge/.style={very thick}}
 	
    \draw[black, very thick] (0,0) circle (1.5);

    \node[] (v0)    at (0,0)            {a};
    \node[] (v1)    at (-1.3747,0.6)    {a};
    \node[] (v2)    at (1.5,0)          {a};
    \node[] (v3)    at (2.5,0)          {a};
    \node[] (v4)    at (0,1.5)          {a};
    \node[] (v5)    at (0,-1.5)         {a};
    \node[] (v6)    at (-1.3747,-0.6)   {a};
    \node[] (a)     at (-2,0.6)         {a};
    \node[] (b)     at (-2,-0.6)        {a};
    \node[] (c)     at (-0.75,0)        {a};
    \node[] (d)     at (3.5,1)          {a};
    \node[] (e)     at (3.5,-1)         {a};
    
    \node[draw = none, fill = none, left=1mm of a] (la) {\large $a$};
    \node[draw = none, fill = none, left=1mm of b] (lb) {\large $b$};
    \node[draw = none, fill = none, below=1mm of c] (lc) {\large $c$};
    \node[draw = none, fill = none, right=1mm of d] (ld) {\large $d$};
    \node[draw = none, fill = none, right=1mm of e] (le) {\large $e$};

    \draw[edge] (v0) -- (c);
    \draw[edge] (v1) -- (a);
    \draw[edge] (v2) -- (v3) node[draw = none, fill = none, above=1mm, midway]{\large $f$};
    \draw[edge, dashed] (v4) -- (v5);
    \draw[edge] (v3) -- (d);
    \draw[edge] (v3) -- (e);
    \draw[edge] (v6) -- (b);

 \end{tikzpicture}
}
\end{minipage}\hfill
\begin{minipage}[c]{.60\textwidth}
 \centering
 \small
 \resizebox{\columnwidth}{!}{
    \begin{tabular}{c|c|c|c|c|c}
            & a         & b             & c             & d             & e  \\
        \hline
        a   & $\{0^1\}$   & $\{3^1,6^2\}$   & $\{4^1,5^1,6^1,7^1\}$   & $\{5^1,6^1,7^1,8^1\}$ & $\{5^1,6^1,7^1,8^1\}$\\
        \hline
        b   &           & $\{0^1\}$       & $\{4^1,5^1,6^1,7^1\}$   & $\{5^1,6^1,7^1,8^1\}$ & $\{5^1,6^1,7^1,8^1\}$\\
        \hline
        c   &           &               & $\{0^1\}$       & $\{5^2,8^2\}$ & $\{5^2,8^2\}$\\
        \hline
        d   &           &               &               & $\{0^1\}$       & $\{2^1\}$\\
        \hline
        e   &           &               &               &               & $\{0^1\}$
    \end{tabular}
    }
 
\end{minipage}\hfill
\caption{A level-$2$ network with its multisets of distances. The network contains two chains~$(a,b)$ and~$(c)$, and a cherry~$\{d,e\}$. All edges incident to leaves are trivial cut-edges, and edge~$f$ is the only cut-edge that is non-trivial.
The dashed path is the side of the blob that contains the leaf~$c$.
In the distance matrix, the diagonal elements are~$\{0\}$, and as the matrix is symmetric, many of the elements are omitted.
The shortest distance matrix can be obtained by taking the smallest element in each multisets to be the element of the matrix in the same position.}
\label{fig:Example}
\end{figure}


\section{Networks that cannot be reconstructed}\label{sec:Negative}


In this section we give examples of networks that cannot be reconstructed from their shortest distances or from their multisets of distances.
Figure~\ref{fig:level2} shows two distinct level-$2$ networks with the same shortest distance matrix.
Observing that we may replace the leaves with the same label by the same pendant subtree to extend this example to a level-$2$ network on at least~$4$ leaves, we obtain the following lemma.

\begin{lemma}\label{lem:L2NSD}
	There exist two distinct level-$2$ networks on~$n$ leaves for~$n\ge4$ with the same shortest distance matrix.
\end{lemma}

\begin{figure}[h]
    \centering
    \resizebox{\columnwidth}{!}{
    \begin{tikzpicture}[every node/.style={draw, circle, fill, inner sep=0pt}]
     	\tikzset{edge/.style={very thick}}
        \begin{scope}[xshift=0cm,yshift=0cm]
        \draw[black, very thick] (0,0) circle (1.5);
        \draw[black, very thick] (-4,0) circle (1.5);

        \node[] (v2)    at (1.3747,0.6)     {a};
        \node[] (v3)    at (1.3747,-0.6)    {a};
        \node[] (v4)    at (-1.5,0)         {a};
        \node[] (u1)    at (-2.5,0)         {a};
        \node[] (u2)    at (-4,1.5)         {a};
        \node[] (u3)    at (-4,-1.5)        {a};
        \node[] (u4)    at (-5.3747,0.6)    {a};
        \node[] (u5)    at (-5.3747,-0.6)   {a};
        
        \node[] (a)     at (-6.3747,0.6)    {a};
        \node[] (b)     at (-6.3747,-0.6)   {a};
        \node[] (c)     at (2.3747,0.6)     {a};
        \node[] (d)     at (2.3747,-0.6)    {a};

        \node[draw = none, fill = none, left=1mm of a]  (la) {\large $a$};
        \node[draw = none, fill = none, left=1mm of b]  (lb) {\large $b$};
        \node[draw = none, fill = none, right=1mm of c] (lc) {\large $c$};
        \node[draw = none, fill = none, right=1mm of d] (ld) {\large $d$};

        \draw[edge] (u2) -- (u3);
        \draw[edge] (u1) -- (v4);
        \draw[edge] (v2) -- (c);
        \draw[edge] (v3) -- (d);
        \draw[edge] (u4) -- (a);
        \draw[edge] (u5) -- (b);
        \end{scope}
        
        \begin{scope}[xshift=12cm,yshift=0cm]
        \draw[black, very thick] (0,0) circle (1.5);
        \draw[black, very thick] (-4,0) circle (1.5);

        \node[] (v2)    at (1.3747,0.6)     {a};
        \node[] (v3)    at (1.3747,-0.6)    {a};
        \node[] (v4)    at (-1.5,0)         {a};
        \node[] (u1)    at (-2.5,0)         {a};
        \node[] (u2)    at (-4,1.5)         {a};
        \node[] (u3)    at (-4,-1.5)        {a};
        \node[] (u4)    at (-5.3747,0.6)    {a};
        \node[] (u5)    at (-5.3747,-0.6)   {a};
        
        \node[] (a)     at (-6.3747,0.6)    {a};
        \node[] (b)     at (-6.3747,-0.6)   {a};
        \node[] (c)     at (2.3747,0.6)     {a};
        \node[] (d)     at (2.3747,-0.6)    {a};

        \node[draw = none, fill = none, left=1mm of a]  (la) {\large $c$};
        \node[draw = none, fill = none, left=1mm of b]  (lb) {\large $d$};
        \node[draw = none, fill = none, right=1mm of c] (lc) {\large $a$};
        \node[draw = none, fill = none, right=1mm of d] (ld) {\large $b$};

        \draw[edge] (u2) -- (u3);
        \draw[edge] (u1) -- (v4);
        \draw[edge] (v2) -- (c);
        \draw[edge] (v3) -- (d);
        \draw[edge] (u4) -- (a);
        \draw[edge] (u5) -- (b);
        \end{scope}

    \end{tikzpicture}}
    \caption{Two level-$2$ networks with the same shortest distances between any pair of leaves.}
    \label{fig:level2}
\end{figure}
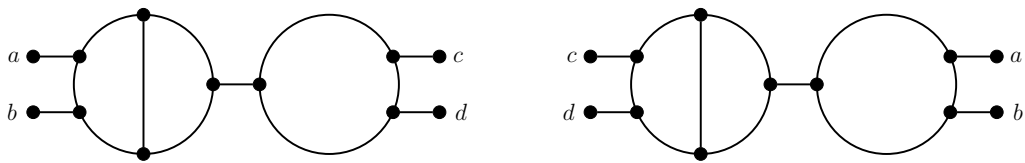

Note that the networks in Figure~\ref{fig:level2} have different multisets of distances -- we investigate this further in Section~\ref{sec:Multiset} and show there that level-$2$ networks are reconstructible from their multisets of distances.

Figure~\ref{fig:level3} presents two level-$3$ networks on~$2$ leaves that have the same multisets of distances.
Because the shortest distance matrix can be obtained by taking the smallest number for each element in the multisets of distances, the two networks also have the same shortest distance matrix.
Observe that this can be generalized to level-$k$ networks for~$k\ge 3$ by replacing the level-$3$ blob by an arbitrary level-$k$ blob.
In addition, applying the same pendant subtree argument as in the level-$2$ network case gives us the following lemma.

\begin{lemma}\label{lem:LkNSM}
	There exist two distinct level-$k$ networks for all~$k\ge3$ with the same shortest distance matrix / multisets of distances.	
\end{lemma}

\begin{figure}[h]
    \centering
    \resizebox{\columnwidth}{!}{
    \begin{tikzpicture}[every node/.style={draw, circle, fill, inner sep=0pt}]
     	\tikzset{edge/.style={very thick}}
        \begin{scope}[xshift=0cm,yshift=0cm]
        \draw[black, very thick] (0,0) circle (1.5);
        \draw[black, very thick] (-4,0) circle (1.5);

        \node[] (v1)    at (1.5,0)          {a};
        \node[] (v2)    at (0,1.5)          {a};
        \node[] (v3)    at (0,-1.5)         {a};
        \node[] (v4)    at (-1.5,0)         {a};
        \node[] (u1)    at (-2.5,0)         {a};
        \node[] (u2)    at (-3.25,1.3)      {a};
        \node[] (u3)    at (-3.25,-1.3)     {a};
        \node[] (u4)    at (-4.75,1.3)      {a};
        \node[] (u5)    at (-4.75,-1.3)     {a};
        \node[] (u6)    at (-5.5,0)         {a};
        
        \node[] (b)     at (2.5,0)          {a};
        \node[] (a)     at (-6.5,0)         {a};

        \node[draw = none, fill = none, right=1mm of b] (lb) {\large $b$};
        \node[draw = none, fill = none, left=1mm of a]  (la) {\large $a$};

        \draw[edge] (u6) -- (a);
        \draw[edge] (v2) -- (v3);
        \draw[edge] (u1) -- (v4);
        \draw[edge] (u2) -- (u3);
        \draw[edge] (u4) -- (u5);
        \draw[edge] (v1) -- (b);
        \end{scope}

        \begin{scope}[xshift=12cm,yshift=0cm]
        \draw[black, very thick] (0,0) circle (1.5);
        \draw[black, very thick] (-4,0) circle (1.5);

        \node[] (v1)    at (1.5,0)          {a};
        \node[] (v2)    at (0,1.5)          {a};
        \node[] (v3)    at (0,-1.5)         {a};
        \node[] (v4)    at (-1.5,0)         {a};
        \node[] (u1)    at (-2.5,0)         {a};
        \node[] (u2)    at (-3.25,1.3)      {a};
        \node[] (u3)    at (-3.25,-1.3)     {a};
        \node[] (u4)    at (-4.75,1.3)      {a};
        \node[] (u5)    at (-4.75,-1.3)     {a};
        \node[] (u6)    at (-5.5,0)         {a};
        
        \node[] (a)     at (2.5,0)          {a};
        \node[] (b)     at (-6.5,0)         {a};

        \node[draw = none, fill = none, right=1mm of a] (la) {\large $a$};
        \node[draw = none, fill = none, left=1mm of b]  (lb) {\large $b$};

        \draw[edge] (v1) -- (a);
        \draw[edge] (v2) -- (v3);
        \draw[edge] (u1) -- (v4);
        \draw[edge] (u2) -- (u3);
        \draw[edge] (u4) -- (u5);
        \draw[edge] (u6) -- (b);
        \end{scope}
    \end{tikzpicture}}
    \caption{Two level-3 networks that have the same shortest distances and the same multisets of distances between any pair of leaves.}
    \label{fig:level3}
\end{figure}
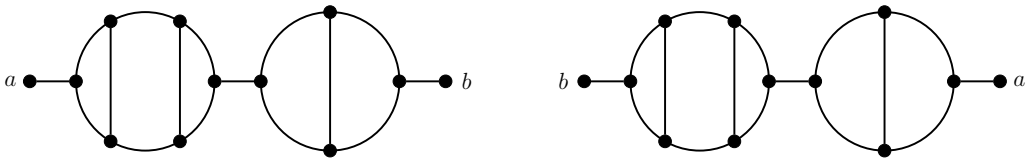

Therefore, networks of level higher than~$1$ are not reconstructible from their shortest distances in general; networks of level higher than~$2$ are not reconstructible from their multisets of distances in general.


\section{Reconstructibility from shortest distances}\label{sec:ShortestDistance}
In this section we show that level-$1$ networks as well as level-$2$ networks on fewer than~$4$ leaves are reconstructible from their shortest distances.
We first look at level-$1$ networks.
Noting that pendant blobs contain exactly one chain, the following lemma shows how we can identify this chain from the shortest distances.

\begin{lemma}\label{lem:L1FindPendantChain}
    Let~$(a_1,\ldots,a_k)$ be a chain of length~$k\geq2$ in a level-1 network.
    Then~$(a_1,\ldots,a_k)$ is contained in a pendant blob if and only if~$d_m(a_1,x) = d_m(a_k,x)$ for all~$x\in X-\{a_1,\ldots,a_k\}$.
\end{lemma}
\begin{proof}
    Suppose first that a chain~$(a_1,\ldots,a_k)$ is contained in a pendant blob $B$.
    Let~$p_1$ and~$p_k$ denote the neighbours of~$a_1$ and~$a_k$ respectively, and let~$q$ denote the common neighbour of~$p_1$ and~$p_k$.
    Let~$x\in X-\{a_1,\ldots,a_k\}$.
    Observe that any shortest path from~$x$ to a leaf contained in~$B$ must pass through the vertex~$q$.
    Therefore we have that
    \begin{equation*}
        d_m(a_1,x) = 2+d_m(q,x) = d_m(a_k,x).
    \end{equation*}
    
    To show the other direction, we prove the contrapositive.
    Suppose that $(a_1,\ldots, a_k)$ is not contained in a pendant blob.
    Then either the chain is incident to cut-edges, or the chain is contained in a non-pendant blob.
    Let~$p_i$ denote the neighbours of~$a_i$ for~$i\in[k]$, and let~$q$ denote the neighbour of~$p_1$ that is not~$a_1$ nor $p_2$.
    Suppose first that the chain is incident to cut-edges.
    Let~$x$ be a leaf in the network that is not on the chain, such that~$x$ is reachable from~$p_1$ via~$p_1q$.
    Then every path between~$x$ and~$a_k$ must pass through the vertices~$p_i$ for~$i\in[k]$, and therefore~$d_m(x,a_k) = d_m(x,a_1) + k-1$.
    Since~$k\geq2$, the equality in the statement of the theorem does not hold.
    
    So now consider the case that the chain is contained in a non-pendant blob.
    Then~$q$ is not a neighbour of~$p_k$; the path between~$q$ and~$p_k$ that does not contain the vertices~$\{p_1,\ldots,p_{k-1}\}$ contains at least three vertices.
    Now let~$x$ be a leaf not on the chain that can be reached from~$q$ via its incident non-trivial cut-edge.
    The shortest path from~$x$ to~$a_1$ and the shortest path from~$x$ to~$a_k$ both contain the shortest path from~$x$ to~$q$.
    By observing that the shortest path from~$q$ to~$a_1$ is shorter than the shortest path between~$q$ and~$a_k$, it follows that~$d_m(x,a_1) < d_m(x,a_k)$.
    Therefore the equality in the statement of the theorem does not hold.
\end{proof}

\begin{theorem}\label{thm:L1ReconstructibleShortPath}
Level-1 networks are reconstructible from their shortest distances.
\end{theorem}
\begin{proof}
First we show that we can recognise cherries, reduce them and change the shortest distances accordingly. Note that as mentioned above, a pair of leaves forms a cherry precisely if their shortest distance is~2.
If there exists a cherry~$\{x,y\}$, we replace it by a leaf~$z$ and set $d_m(z,a) := d_m(x,a) - 1$ for all~$a\in X-\{x,y\}$. All other shortest distances between leaf-pairs remain unchanged.
After reconstructing the network from the modified distance matrix, we replace the leaf~$z$ by a cherry on~$\{x,y\}$.
So, without loss of generality, we assume from now on that there are no cherries.

We now consider the case that there is exactly one blob. 
Since there are no cherries, all leaves are contained in this blob.
We can recognize this by seeing that there is a chain $(a_1,\ldots, a_k)$ of length~$k\geq 3$ that satisfies $d_m(a_1,a_k)=3$. This immediately shows how to reconstruct level-$1$ networks that contain exactly one blob. Hence, we assume from now on that there are at least two blobs.

Note that pendant blobs must contain a chain of length at least~$2$ since networks do not contain parallel edges.
By Lemma~\ref{lem:L1FindPendantChain}, we can find chains on pendant blobs.
We reduce a chain~$(a_1,\ldots,a_k)$ contained in a pendant blob by replacing the blob by a leaf~$z$ and setting $d_m(x,z) := d_m(x,a_1) - 2$ for all~$x\in X - 
\{a_1,\ldots ,a_k\}$.
All shortest distances between other leaf-pairs remain unchanged, since their paths do not travel through pendant blobs.
It is again easy to reconstruct the blob after reconstructing the reduced network, since we know that $(a_1,\ldots ,a_k)$ must form a chain on the blob, in that order.

This finishes the proof of the theorem since any level-1 network has a cherry, a pendant blob, or exactly one blob.
\end{proof}

We note that the restriction of Theorem~\ref{thm:L1ReconstructibleShortPath} to networks without triangles also follows from Theorem~5 of~\cite{hayamizu2020recognizing}.
We give the proof above to account for the triangle case and to give a more direct graph-theoretical proof that is independent of the results provided by Hayamizu et al..
Observe that trees (level-$0$ networks) are also level-$1$ networks.
Thus Theorem~\ref{thm:L1ReconstructibleShortPath} gives the following corollary, which we include here for completeness.
This is a classical result that was proven in~\cite{hakimi1965distance}.
\begin{corollary}
    Trees are reconstructible from their shortest distances.
\end{corollary}
Next, we show that level-$2$ networks on fewer than~$4$ leaves are also reconstructible from their shortest distances.

\begin{lemma}\label{lem:L2<3ShortestRecon}
    Level-2 networks on~$X$ for~$|X|\leq 3$ are reconstructible from their shortest distances.
\end{lemma}
\begin{proof}
    There can only be one network on a single taxon, namely the singleton graph. Such a graph is trivially reconstructible from its shortest distances.
    So suppose that~$|X| = 2$, say~$X = \{x,y\}$, and let~$N$ be a network on~$X$.
    Below, we will prove the claim that~$N$ consists only of level-$2$ blobs, where each level-$2$ blob is incident to exactly two cut-edges.
    In particular,~$N$ contains at most two pendant blobs, one of which contains the neighbour of~$x$ and the other the neighbour of~$y$.
    Since each additional level-$2$ blob increases the shortest distance between~$x$ and~$y$ by~$3$, it follows that~$d_m(x,y) = 3k+1$ where~$k$ denotes the number of level-$2$ blobs in~$N$.
    From there, it follows that~$N$ is reconstructible from its shortest distances.
    
    We now prove the claim.
    Note first that every blob in~$N$ must be incident to exactly two cut-edges.
    A blob cannot be incident to only one cut-edge. If the blob is level-$1$ then this would imply that it contains a loop; if the blob is level-$2$ then this would imply that it contains parallel edges.
    This also implies that every pendant blob must be incident to at least one trivial cut-edge.
    On the other hand if a blob is incident to more than two cut-edges, say~$c$ cut-edges, then this would imply that the network contains at least~$c$ pendant blobs.
    Since every pendant blob must be incident to at least one trivial cut-edge, this implies that the network contains at least~$c>2$ leaves, which is a contradiction.
    Therefore every blob in~$N$ must be incident to exactly two cut-edges.
    Now observe that a level-$1$ blob that is incident to exactly two cut-edges contains parallel edges.
    It follows that every blob in~$N$ must be a level-$2$ blob that is incident to exactly two cut-edges.
    This proves the claim, from which it follows by the argument presented above that~$N$ is reconstructible from its shortest distances for~$|X|=2$.
    
    Suppose now that~$|X|=3$, and let~$X = \{x,y,z\}$.
    Here we consider $BT(N)$, the \emph{blob-tree} of~$N$, which is obtained from~$N$ by replacing each blob of~$N$ by a single vertex.
    Since~$|X|=3$,~$BT(N)$ contains exactly one vertex of degree-$3$, three vertices of degree-$1$ (which are the leaves~$x,y,$ and~$z$), and all other vertices are of degree-$2$.
    By a similar argument as presented in the~$|X|=2$ case, the degree-$2$ vertices of~$BT(N)$ correspond to level-$2$ blobs.
    The degree-$3$ vertex could be an internal vertex of the network, a level-$1$ blob, or a level-$2$ blob.
    In the case that it is a level-$2$ blob, there are two possibilities.
    Either the three edges are incident to different sides of the blob, or two edges are incident to the same side of the blob and the third edge to another side.
    See Figure~\ref{fig:L2<3ShortestRecon} for these four possibilities.
    Observe that these four possibilities all contribute different distance lengths to inter-taxa distances.
    In particular, we have that the degree-$3$ vertex is a (an)
    \begin{itemize}
        \item internal vertex if and only if
        \[(d(x,y), d(y,z), d(x,z)) = (2(\bmod{3}),2(\bmod{3}),2(\bmod{3}));\]
        \item level-$1$ blob if and only if
        \[(d(x,y), d(y,z), d(x,z)) = (0(\bmod{3}),0(\bmod{3}),0(\bmod{3}));\]
        \item level-$2$ blob with all edges on different sides if and only if
        \[(d(x,y), d(y,z), d(x,z)) = (1(\bmod{3}),1(\bmod{3}),1(\bmod{3}));\]
        \item level-$2$ blob with the two edges that lead to leaves~$x$ and~$y$ on the same side if and only if
        \[(d(x,y), d(y,z), d(x,z)) = (0(\bmod{3}),1(\bmod{3}),1(\bmod{3})).\]
    \end{itemize}
    Therefore we may identify the blob corresponding to the degree-$3$ vertex of the blob-tree by taking the distances modulo~$3$.
    
    To finish the proof, take two networks~$N,N'$ with the same shortest distance matrix.
    By the previous paragraph, we may assume that~$N$ and~$N'$ have the same blob corresponding to the degree-$3$ vertex of their blob-trees.
    Assume that~$N\neq N'$.
    Then the two blob-trees~$BT(N)$ and~$BT(N')$ are different.
    Note that the shortest distances are determined by the number of degree-$2$ vertices between leaves in the blob-tree.
    Since~$\cD_m(N) = \cD_m(N')$, we have that the number of degree-$2$ vertices between two leaves, say~$x$ and~$y$, is the same in both~$BT(N)$ and~$BT(N')$.
    However since~$BT(N)$ differs from~$BT(N')$, the positioning of the degree-$3$ vertex must differ.
    But this would imply that upon placing~$z$ together with some degree-$2$ vertices, we can only satisfy one of~$d^N_m(x,z) = d^{N'}_m(x,z)$ or~$d^N_m(y,z) = d^{N'}_m(y,z)$.
    This contradicts the assumption that~$\cD_m(N) = \cD_m(N')$.
    Therefore we must have~$N=N'$, and that level-$2$ networks on~$X$ for~$|X|=3$ are reconstructible from their shortest distances.
\end{proof}

\begin{figure}
    \begin{subfigure}{.25\textwidth}
    \centering
        \begin{tikzpicture}[every node/.style={draw, circle, fill, inner sep=0pt}]
         	\tikzset{edge/.style={very thick}}
         	
         	\node[] (origin)    at (0,0)    {a};
         	\node[] (x)         at (180:1)  {a};
         	\node[] (y)         at (60:1)   {a};
         	\node[] (z)         at (300:1)  {a};
         	
            \node[draw = none, fill = none, left=1mm of x]  (lx) {\large $x$};
            \node[draw = none, fill = none, right=1mm of y]  (ly) {\large $y$};
            \node[draw = none, fill = none, right=1mm of z]  (lz) {\large $z$};
            
            \draw[edge, dashed] (x) -- (origin);
            \draw[edge, dashed] (y) -- (origin);
            \draw[edge, dashed] (z) -- (origin);
        \end{tikzpicture}
    \caption{}
    \end{subfigure}\hfill
    \begin{subfigure}{.25\textwidth}
    \centering
        \begin{tikzpicture}[every node/.style={draw, circle, fill, inner sep=0pt}]
         	\tikzset{edge/.style={very thick}}
         	
         	\draw[black, very thick] (0,0) circle (0.4);
         	
         	\node[] (bx)        at (180:0.4)    {a};
         	\node[] (by)        at (60:0.4)     {a};
         	\node[] (bz)        at (300:0.4)    {a};
         	
         	\node[] (x)         at (180:1)  {a};
         	\node[] (y)         at (60:1)   {a};
         	\node[] (z)         at (300:1)  {a};
         	
            \node[draw = none, fill = none, left=1mm of x]  (lx) {\large $x$};
            \node[draw = none, fill = none, right=1mm of y]  (ly) {\large $y$};
            \node[draw = none, fill = none, right=1mm of z]  (lz) {\large $z$};
         	
         	\draw[edge, dashed] (bx) -- (x);
         	\draw[edge, dashed] (by) -- (y);
         	\draw[edge, dashed] (bz) -- (z);
         
        \end{tikzpicture}
    \caption{}
    \end{subfigure}\hfill
    \begin{subfigure}{.25\textwidth}
    \centering
        \begin{tikzpicture}[every node/.style={draw, circle, fill, inner sep=0pt}]
         	\tikzset{edge/.style={very thick}}
         	
         	\draw[black, very thick] (0,0) circle (1);
         	
         	\node[] (origin)    at (0,0)        {a};
         	\node[] (north)     at (90:1)       {a};
         	\node[] (south)     at (270:1)      {a};
         	
         	\node[] (bx)        at (180:1)      {a};
         	\node[] (bz)        at (0:1)        {a};
         	
         	\node[] (x)         at (-1.5,0)     {a};
         	\node[] (y)         at (-0.5,0)     {a};
         	\node[] (z)         at (0.5,0)      {a};
         	
            \node[draw = none, fill = none, left=1mm of x]  (lx) {\large $x$};
            \node[draw = none, fill = none, below=1mm of y]  (ly) {\large $y$};
            \node[draw = none, fill = none, below=1mm of z]  (lz) {\large $z$};
         	
         	\draw[edge] (north) -- (south);
         	
         	\draw[edge, dashed] (bx) -- (x);
         	\draw[edge, dashed] (origin) -- (y);
         	\draw[edge, dashed] (bz) -- (z);
         
        \end{tikzpicture}
    \caption{}
    \end{subfigure}\hfill
    \begin{subfigure}{.25\textwidth}
    \centering
        \begin{tikzpicture}[every node/.style={draw, circle, fill, inner sep=0pt}]
         	\tikzset{edge/.style={very thick}}
         	
         	\draw[black, very thick] (0,0) circle (1);
         	
         	\node[] (north)     at (90:1)   {a};
         	\node[] (south)     at (270:1)  {a};
         	
         	\node[] (bx)        at (180:1)  {a};
         	\node[] (by)        at (30:1)   {a};
         	\node[] (bz)        at (330:1)  {a};
         	
         	\node[] (x)         at (180:0.5){a};
         	\node[right = 5mm of by] (y)    {a};
         	\node[right = 5mm of bz] (z)    {a};
         	
            \node[draw = none, fill = none, below=1mm of x]  (lx) {\large $x$};
            \node[draw = none, fill = none, right=1mm of y]  (ly) {\large $y$};
            \node[draw = none, fill = none, right=1mm of z]  (lz) {\large $z$};
         	
         	\draw[edge] (north) -- (south);
         	
         	\draw[edge, dashed] (bx) -- (x);
         	\draw[edge, dashed] (by) -- (y);
         	\draw[edge, dashed] (bz) -- (z);
         
        \end{tikzpicture}
    \caption{}
    \end{subfigure}
    \caption{The four possible degree-$3$ vertices in the blob-tree of a level-$2$ network on three leaves~$\{x,y,z\}$.
    (a) An internal vertex.
    (b) A level-$1$ blob.
    (c) A level-$2$ blob with all leaves reachable from different sides of the blob.
    (d) A level-$2$ blob where~$y$ and~$z$ are reachable from the same side of the blob.
    The dashed lines can be replaced by paths that contain any number of level-$2$ blobs.
    This is possible because we take the distances modulo~$3$ and since each additional level-$2$ blob contributes an extra length-$3$ to the shortest inter-taxa distance.}
    \label{fig:L2<3ShortestRecon}
\end{figure}
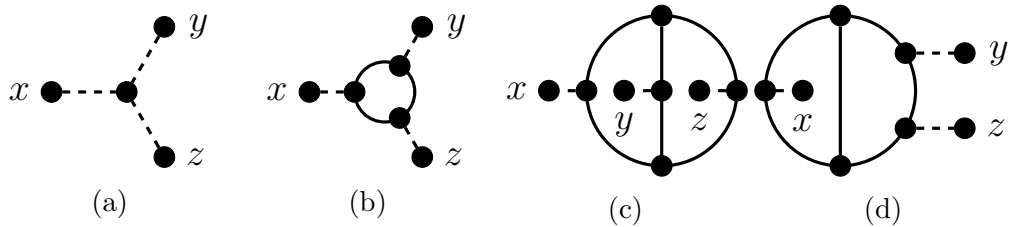


\section{Reconstructibility of level-2 networks from their multisets of distances}\label{sec:Multiset}

In the last two sections, we showed that level-$1$ networks are reconstructible from their shortest distances, level-$k$ networks for~$k\geq 2$ are in general not reconstructible from their shortest distances, and level-$k$ networks for~$k\ge3$ are in general not reconstructible from their multisets of distances.
In this section, we investigate the remaining case, and show that level-$2$ networks are reconstructible from their multisets of distances.
The main theorem is the following.

\begin{theorem}\label{thm:L2ReconstructibleMultiset}
    Level-2 networks are reconstructible from their multisets of distances.
\end{theorem}

The key ideas in proving the theorem are as follows.
We first identify and reduce all cherries of the network.
To identify cherries we observe that two leaves~$x$ and~$y$ form a cherry if and only if~$d(x,y) = \{2\}$.
To reduce cherries we replace it by a new leaf~$z$ and adjust the distance matrix accordingly, as done for the level-$1$ networks in the proof of Theorem~\ref{thm:L1ReconstructibleShortPath}.
Next, we identify all leaves that are not contained in blobs, delete those leaves, and adjust the distance matrix accordingly.
We show that each leaf that is deleted in this manner can be reattached to the reduced network in a unique fashion.
After applying these two reductions, two chains are adjacent if and only if they are contained in the same blob.
Using this observation, we then show that it is possible to identify pendant blobs, replace them by a new leaf, and adjust the distance matrix accordingly.
Continuing in this fashion, we eventually reach the situation when the reduced network contains exactly one blob. 
We show that networks on single blobs are reconstructible from their multisets of distances, at which point it follows that simply reversing the reduction steps taken yields the original network.

We start with the two easy cases, when the network contains a cherry or a single blob.
\begin{observation}\label{obs:Cherry}
    Let~$N$ be a level-$2$ network on~$X$ and suppose that leaves~$x$ and~$y$ form a cherry in~$N$.
    Upon replacing the cherry by a leaf~$z$, we obtain a network~$N'$ on~$X'=X\cup\{z\}-\{x,y\}$ such that the multisets of distances for~$N'$ contains the elements
    \[ d^{N'}(a,b) = 
    \begin{cases} 
        d^N(a,b)      & \text{if }a,b\in X-\{x,y\} \\
        d^N(a,x) - 1  & \text{if }a\in X-\{x,y\} \text{ and }b = z.
    \end{cases}
    \]
    One may obtain~$N$ from~$N'$ by replacing the leaf~$z$ by a cherry~$\{x,y\}$.
\end{observation}

\begin{lemma}\label{lem:L2 1BSD}
    Level-$2$ networks containing a single blob are reconstructible from their shortest distances.
\end{lemma}
\begin{proof}
    Let~$N$ be a level-$2$ network containing a single blob.
    Assume without loss of generality that~$N$ contains no cherries, as we can recognize them from the shortest distances and reduce them by Observation~\ref{obs:Cherry}.
    If~$N$ is a level-$1$ blob then we may reconstruct it from shortest distances by Theorem~\ref{thm:L1ReconstructibleShortPath}.
    If~$N$ is a level-$2$ blob then the blob must contain at least two chains since it has no parallel edges, and at most three chains.
    Noting that chains can be identified from the shortest distances, the placement of the chains on the blob sides can be done by matching the end-leaves of chains that have shortest distance~$4$.
\end{proof}

\subsection{Leaves not contained in blobs}

\begin{lemma}\label{lem:LeafBlobAdjacency}
    Let~$N$ be a level-$2$ network on~$X$ where~$|X|\geq3$.
    A leaf~$x$ is not contained in a blob if and only if there exists a unique partition~$Y\cup Z$ of~$X-\{x\}$ such that~$Y,Z\neq \emptyset$ and~$d_m(y,z) = d_m(x,y) + d_m(x,z) - 2$ for all~$y\in Y$ and~$z\in Z$.
\end{lemma}
\begin{proof}
    Suppose first that a leaf~$x$ is not contained in a blob.
    Let~$p_x$ denote the neighbour of~$x$, and let~$p,q$ denote the two neighbours of~$p_x$ that is not~$x$.
    Observe that every leaf in~$X-\{x\}$ can be reached from~$p_x$ via one of the cut-edges~$p_xp$ or~$p_xq$.
    Let~$Y$ and~$Z$ denote the set of all leaves that can be reached from~$p_x$ via the cut-edge~$p_xp$ and~$p_xq$, respectively. 
    Note that a shortest path between some~$y\in Y$ and some~$z\in Z$ passes through the edges~$p_xp$ and~$p_xq$.
    Then by observing that the shortest path from~$x$ to~$y$ and the shortest path from~$x$ to~$z$ uses the same edges as the shortest path from~$y$ to~$z$, bar the use of the edge incident to~$x$ twice, we obtain the equation~$d_m(y,z) = d_m(x,y) + d_m(x,z) - 2$ for all~$y\in Y$ and~$z\in Z$.
    
    We now show that such a partition is unique.
    We claim that all leaves that can be reached from~$p_x$ via the edge~$p_xp$ must be contained in the same set in the partition.
    Let~$y_1$ and~$y_2$ be an arbitrarily chosen pair of leaves that can be reached from~$p_x$ via the edge~$p_xp$, and suppose for a contradiction that they are placed in different sets of the partition.
    Then,
    \begin{align*}
        d_m(x,y_1) + d_m(x,y_2) - 2 &= d_m(p,y_1) + d_m(p,y_2) + 2 \\
                                    &> d_m(p,y_1) + d_m(p,y_2) \\
                                    &\geq d_m(y_1,y_2),
    \end{align*}
    where the final inequality is the triangle inequality.
    Hence~$y_1$ and~$y_2$ must be contained in the same set of the partition; since~$y_1$ and~$y_2$ were chosen arbitrarily, all leaves that can be reached from~$p_x$ via the edge~$p_xp$ must be contained in the same set in the partition.
    Similarly, all leaves that can be reached from~$p_x$ via the edge~$p_xq$ must be contained in the same set in the partition.
    Observe that all leaves in~$X-\{x\}$ can be reached from~$p_x$ via the edge~$p_xp$ or via the edge~$p_xq$.
    Since neither sets of the partition can be empty, it follows then that the partition must be unique, with~$Y$ and~$Z$ containing all leaves that can be reached from~$p_x$ via~$p_xp$ and~$p_xq$, respectively.
    \medskip
    
    To prove the other direction, we show that if a leaf~$x$ is contained in a blob~$B$, then there is no such partition that satisfies the given equation.
    Let~$p_x$ denote the neighbour of~$x$.
    We first show that for leaves~$y,z\in X-\{x\}$, if all shortest paths between~$y$ and~$z$ do not contain the vertex~$p_x$, then the equation is not satisfied by~$y$ and~$z$.
    Let~$p_y$ and~$p_z$ denote the vertices on~$B$ that are closest to the leaves~$y$ and~$z$ respectively. 
    Note that it is possible to have~$p_y = p_z$ -- this is the case where all shortest paths between~$y$ and~$z$ do not pass through~$B$.
    Then the following equations hold:
    \begin{align*}
        d_m(x,y) &= 1 + d_m(p_x,p_y) + d_m(p_y,y)\\
        d_m(x,z) &= 1 + d_m(p_x,p_z) + d_m(p_z,z).
    \end{align*}
    We now distinguish two cases.
    \begin{enumerate}
        \item If~$p_y\neq p_z$, then by the triangle inequality and as all shortest paths between~$y$ and~$z$ do not contain the vertex~$p_x$, we must have that
        \begin{equation}\label{eqn:Lvl1BlobAdjacencyp_y_neq_p_z}
            d_m(p_y,p_z) < d_m(p_x,p_y) + d_m(p_x,p_z).
        \end{equation}
        It follows that
        \begin{align*}
            d_m(y,z)    &= d_m(y,p_y) + d_m(p_y,p_z) + d_m(p_z,z)\\
                        &= \left(d_m(x,y) - d_m(p_x,p_y) - 1\right) + d_m(p_y,p_z)\\ 
                        &\quad\quad + \left(d_m(x,z) - d_m(p_x,p_z) - 1\right)\\
                        &= d_m(x,y) + d_m(x,z) - 2 + d_m(p_y,p_z)\\ 
                        &\quad\quad - \left(d_m(p_x,p_y) + d_m(p_x,p_z)\right)\\
                        &< d_m(x,y) + d_m(x,z) - 2,
        \end{align*}
        where the final inequality follows from Inequality~\ref{eqn:Lvl1BlobAdjacencyp_y_neq_p_z}.
        \item If~$p_y = p_z$, then let~$p$ denote the neighbour of~$p_y$ that is not on the blob~$B$. 
        Then
        \begin{align*}
            d_m(y,z)    &\leq d_m(y,p_y) + d_m(z,p_y) - 2d_m(p_y,p)\\
                        &= (d_m(x,y) - d_m(p_x,p_y) - 1) + (d_m(x,z) - d_m(p_x,p_y) - 1) - 2\\
                        &= d_m(x,y) + d_m(x,z) - 2 - 2d_m(p_x,p_y) - 2\\
                        &< d_m(x,y) + d_m(x,z) - 2,
        \end{align*}
        where the first inequality follows since the shortest path between~$y$ and~$z$ may not pass through~$p$ (e.g., if~$p$ is a vertex on a blob), and the final inequality follows as~$d_m(p_x,p_y)\geq1$ and~$d_m(p_y,p)=1$.
    \end{enumerate}
    It remains to show that for any partition~$Y\cup Z$ of~$X-\{x\}$ where~$Y,Z\neq\emptyset$, there exists a leaf pair~$y\in Y$ and~$z\in Z$ such that no shortest path between~$y$ and~$z$ uses~$p_x$.
    
    Suppose first that $B$ is a level-$1$ blob.
    Since our network contains no parallel edges,~$B$ must be incident to at least two cut-edges in addition to the edge~$p_xx$.
    If two leaves that can be reached from~$B$ via the same cut-edge are placed in different sets of the partition, then we are done as no shortest path between these leaves uses~$p_x$; therefore we may assume that leaves that can be reached from~$B$ via the same cut-edge are placed in the same set in the partition.
    Since~$Y$ and~$Z$ are both non-empty, there must exist two cut-edges~$e_1, e_2$ (excluding~$p_xx$) whose endpoints form an edge of~$B$, such that there exists a leaf that can be reached from~$B$ via~$e_1$ and a leaf that can be reached from~$B$ via~$e_2$ for which the two leaves lie in different sets of the partition.
    Every shortest path between these two leaves passes through the edge connecting the endpoints of~$e_1$ and~$e_2$ and therefore does not use~$p_x$.
    Therefore we are done.

    Now suppose that~$B$ is a level-$2$ blob.
    For the same reason as in the level-1 case (see proof of Theorem~\ref{thm:L1ReconstructibleShortPath}), if there are two leaves that can be reached from~$B$ via the same cut-edge that are placed in different sets of the partition, then we are done; therefore we may assume that leaves that can be reached from~$B$ via the same cut-edge are placed in the same set in the partition.
    Since~$Y$ and~$Z$ are both non-empty, it follows that there exist two cut-edges~$e_1,e_2$ incident to~$B$, such that leaves~$y,z$ can be reached from~$B$ via~$e_1,e_2$, respectively, for which~$y\in Y$ and~$z\in Z$.
    There must exist a pair of such cut-edges such that all shortest paths between their endpoints on~$B$ do not contain~$p_x$, since there exist enough cut-edges to ensure there are no parallel edges in~$B$.
    Given such a pair of cut-edges, take one leaf that can be reached from~$B$ via the first cut-edge and take another leaf that can be reached from~$B$ via the other cut-edge.
    Then no shortest path between this pair of leaves uses~$p_x$, and thus we are done.
\end{proof}

Lemma~\ref{lem:LeafBlobAdjacency} does not hold in general for networks of level higher than~$2$.
An example of this for a level-$3$ network is shown in Figure~\ref{fig:LeafBlobAdjacencyCE}.

We now show that after identifying a leaf that is not contained in a blob, we can delete it from the network and adjust the distance matrix accordingly.
We also show that upon reconstructing the reduced network from the modified distance matrix, there is a unique cut-edge to which we may \emph{reattach} the deleted leaf.
\emph{Reattaching} a leaf~$x$ to a cut-edge is the action of subdividing the cut-edge by a vertex~$p_x$, and adding an edge~$p_xx$.
In the setting of Lemma~\ref{lem:LeafBlobAdjacency}, we say that the unique partition~$Y\cup Z$ is \emph{induced} by the leaf~$x$.

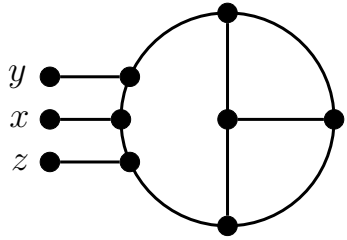
\begin{figure}
    \centering
    \resizebox{0.35\textwidth}{!}{
    \begin{tikzpicture}[every node/.style={draw, circle, fill, inner sep=0pt}]
     	\tikzset{edge/.style={very thick}}
     	
        \draw[black, very thick] (0,0) circle (1.5);
    
        \node[] (c)     at (0,0)            {a};
        \node[] (w)     at (-1.5,0)         {a};
        \node[] (n)     at (0,1.5)          {a};
        \node[] (e)     at (1.5,0)          {a};
        \node[] (s)     at (0,-1.5)         {a};
        \node[] (py)    at (-1.3747,0.6)    {a};
        \node[] (pz)    at (-1.3747,-0.6)   {a};
        \node[] (x)     at (-2.5,0)         {a};
        \node[] (y)     at (-2.5,0.6)       {a};
        \node[] (z)     at (-2.5,-0.6)      {a};

        \node[draw = none, fill = none, left=1mm of x] (lx) {\large $x$};
        \node[draw = none, fill = none, left=1mm of y] (ly) {\large $y$};
        \node[draw = none, fill = none, left=1mm of z] (lz) {\large $z$};
    
        \draw[edge] (c) -- (e);
        \draw[edge] (n) -- (s);
        \draw[edge] (py) -- (y);
        \draw[edge] (pz) -- (z);
        \draw[edge] (w) -- (x);
    
    \end{tikzpicture}
    }
    \caption{A level-$3$ network on~$X = \{x,y,z\}$ where all of its leaves are contained in a blob.
    $Y=\{y\}$ and~$Z=\{z\}$ is a partition of~$X -\{x\}$ such that~$Y,Z\neq\emptyset$ and~$d_m(y,z) = d_m(x,y) + d_m(x,z)-2$ for all~$y\in Y$ and~$z\in Z$.
    Observe that this holds in general for level-$k$ networks where~$k\ge3$ by replacing the level-$3$ blob by an arbitrary level-$k$ blob.}
    \label{fig:LeafBlobAdjacencyCE}
\end{figure}

\begin{lemma}\label{lem:LeafBlobAdjacencyAdjust}
    Let~$N$ be a level-$2$ network on~$X$ where~$|X|\ge3$, and let~$x$ be a leaf that is not contained in a blob.
    Let~$Y\cup Z$ denote the unique partition of~$X' = X-\{x\}$ that is induced by~$x$.
    Then upon deleting the leaf~$x$, we obtain a network~$N'$ on~$X'$ such that the multisets of distances for~$N'$ contains the elements
    \[ d^{N'}(y,z) = 
    \begin{cases} 
        d^N(y,z)      & \text{if }y,z\in Y \text{ or } y,z\in Z \\
        d^N(y,z) - 1  & \text{if }y\in Y, z\in Z \text{ or } z\in Y, y\in Z.
    \end{cases}
    \]
    In addition, there is only one edge location in~$N'$ where~$x$ can be reattached to, to obtain a network with the same multisets of distances as~$N$.
    In particular, this network is isomorphic to~$N$.
\end{lemma}

\begin{proof}
    Let~$p_x$ be the neighbour of~$x$ in~$N$, and let~$p$ and~$q$ be the other neighbours of~$p_x$ that are not~$x$.
    As shown in the proof of Lemma~\ref{lem:LeafBlobAdjacency}, the sets~$Y$ and~$Z$ correspond to the leaves that can be reached from~$p_x$ via~$p_xp$ and via~$p_xq$, respectively.
    Upon deleting~$x$ from~$N$, we note that~$p_x$ becomes a vertex of degree-$2$ and is therefore suppressed in the resulting subgraph.
    Then all paths in~$N$ that used the edge~$p_xp$ and the edge~$p_xq$ have their length decreased by~$1$ in~$N'$; all paths in~$N$ that did not use the edges~$p_xp$ and~$p_xq$ are unaffected by this vertex suppression.
    Observe that any path between a leaf in~$Y$ and a leaf in~$Z$ uses the edges~$p_xp, p_xq$ in~$N$.
    Furthermore, any path between two leaves in~$Y$ or any path between two leaves in~$Z$ did not use the edges~$p_xp, p_xq$ in~$N$.
    Therefore the multisets of distances of~$N'$ can be obtained from the multisets of distances of~$N$ as shown in the statement of the lemma.
    
    We now prove the second statement, namely that~$N'$ contains only one edge where~$x$ can be reattached to, so as to obtain a network with the same multisets of distances as~$N$.
    By Lemma~\ref{lem:LeafBlobAdjacency}, we know that~$x$ is not in a blob, and that~$x$ induces a partition~$Y\cup Z$ of~$X'$.
    This implies that~$x$ must be reattached to~$N'$ at a cut-edge that induces the partition~$Y\cup Z$.
    We now show that there is only one such cut-edge in~$N'$ if we are to obtain a network with the same multisets of distances as~$N$ upon reattaching~$x$.
    If there are two cut-edges~$e_1,e_2$ in~$N'$ that induce the same required partition~$Y\cup Z$, observe that any path from~$e_1$ to~$e_2$ must consist only of level-$2$ blobs that are incident to exactly two cut-edges.
    Note that level-$1$ blobs cannot be included here as otherwise we would produce parallel edges.
    Now take any leaf~$y\in X-\{x\}$, and let~$N_1$ and~$N_2$ denote the networks obtained by attaching~$x$ to~$e_1$ and~$e_2$ respectively.
    Because of the level-$2$ blobs between~$e_1$ and~$e_2$, we have that~$d_m^{N_1}(x,y) \neq d_m^{N_2}(x,y)$.
    But we know that there must exist one cut-edge~$e$ in~$N'$ to which we can attach~$x$ to obtain~$N$.
    We locate this edge~$e$ by finding one that induces the correct partition and satisfies the equation~$d_m^{N_e}(x,y) = d_m^N(x,y)$.
    This proves the claim that~$x$ can be added back to~$N'$ via a unique edge to obtain a network with the same multisets of distances as~$N$.
    Since there is a unique edge where~$x$ can be attached to in order to obtain a network with the same multisets of distances as~$N$, the network obtained this way must be isomorphic to~$N$.
\end{proof}

\subsection{Pendant blobs}

For the remainder of this section, we will restrict to level-$2$ networks with at least two blobs and in which all leaves are contained in blobs.
We can do this by Observation~\ref{obs:Cherry} and Lemmas~\ref{lem:L2 1BSD},~\ref{lem:LeafBlobAdjacency}, and~\ref{lem:LeafBlobAdjacencyAdjust}.

\subsubsection{Pendant level-\emph{1} blobs}

\begin{lemma}\label{lem:L2NetPendantL1Blob}
    Let~$N$ be a level-$2$ network on~$X$.
    A chain~$(a_1,\ldots,a_k)$ with~$k\geq2$ is contained in a pendant level-$1$ blob if and only if~$d(a_1,a_k) = \{4^1,(k+1)^1\}$.
\end{lemma}
\begin{proof}
    Suppose first that a chain~$(a_1,\ldots,a_k)$ with~$k\geq2$ is contained in a pendant level-$1$ blob~$B$.
    As there is only one non-trivial cut-edge incident to~$B$, this chain is the only chain that is contained in~$B$.
    It is then clear that, we must have~$d(a_1,a_k) = \{4^1,(k+1)^1\}$.
    \medskip
    
    Now suppose that there exists a chain~$(a_1,\ldots,a_k)$ with~$k\geq2$ such that~$d(a_1,a_k)=\{4^1,(k+1)^1\}$.
    Clearly the distance~$k+1$ corresponds to the path between~$a_1$ and~$a_k$ that passes through the neighbours of~$a_i$ for~$i\in[k]$.
    Therefore we examine the path between~$a_1$ and~$a_k$ that does not pass through the neighbours of~$a_{i+1}$ for~$i\in[k-2]$.
    Note first that the chain cannot be contained in a non-pendant level-$1$ blob, as otherwise this path between~$a_1$ and~$a_k$ would pass through at least two vertices that are incident to non-trivial cut-edges. In this case, the length of the path between~$a_1$ and~$a_k$ would be at least~$5$, which is a contradiction.
    The chain also cannot be contained in a level-$2$ blob, as otherwise the set~$d(a_1,a_k)$ would contain at least~$3$ elements.
    Therefore the chain must be contained in a pendant level-$1$ blob.
\end{proof}

\begin{lemma}\label{lem:L2NetPendantL1BlobAdjust}
    Let~$N$ be a level-$2$ network on~$X$ in which~$(a_1,\ldots,a_k)$ is a chain that is contained in a pendant level-$1$ blob.
    Let~$N'$ be the network on~$X' = X\cup\{z\}-\{a_1,\ldots,a_k\}$ obtained from~$N$ by replacing the pendant blob by a leaf~$z$.
    For every~$x\in X'-\{z\}$, we can uniquely partition the multiset of distances~$d^N(x,a_1)$ into two equal sized sets~$A$ and~$B$ such that~$A-2 = B-(k+1)$. Then the multisets of distances of~$N'$ contains the elements
    \[d^{N'}(x,y) =
    \begin{cases} 
        d^N(x,y)      & \text{if }x,y\in X'-\{z\} \\
        A-2         & \text{if }y=z.
    \end{cases}
    \]
\end{lemma}
\begin{proof}
    We first prove the claim that for every~$x\in X'-\{z\}$, we can uniquely partition the multiset of distances~$d(x,a_1)$ into two equal sized sets~$A$ and~$B$ such that~$A-2 = B-(k+1)$.
    As usual, let~$p_i$ denote the neighbours of~$a_i$ for~$i\in [k]$, and let~$q$ denote the neighbour of~$p_1$ that is not~$a_1$ nor~$p_2$.
    Note that~$k\ge2$ since otherwise there would be parallel edges.
    Let~$x\in X'$. Then any path from~$x$ to~$a_1$ consists of a path from~$x$ to~$q$ and a path from~$q$ to~$a_1$.
    There are two possible paths from~$q$ to~$a_1$: one is of length~$2$ and uses the edges~$qp_1, p_1a_1$; the other is of length~$k+1$ and uses the edges~$qp_k, p_kp_{k-1},\ldots,p_2p_1,p_1a_1$.
    Therefore every path from~$x$ to~$q$ yields two paths from~$x$ to~$a_1$, for which one of the paths is longer than the other by a length of~$k-1$.
    This implies that the size of the multiset~$d(x,a_1)$ is even, since every path from~$x$ to~$a_1$ can be matched to another path from~$x$ to~$a_1$ that shares the same part of the path between~$x$ and~$q$.
    Now take the smallest element~$d\in d(x,a_1)$. By the argument presented above, there must exist a corresponding element~$d+k-1\in d(x,a_1)$.
    We place~$d$ in set~$A$ and we place~$d+k-1$ in set~$B$, remove both elements from~$d(x,a_1)$ and recurse.
    By continuing this for the smallest element in~$d(x,a_1)$ at each step, this partitions the multiset into a bipartition~$d(x,a_1) = A\cup B$ where~$|A| = |B| = d(x,a_1)/2$, such that~$A + (k-1) = B$.
    It follows from iteratively adding the smallest element from~$d(x,a_1)$ to~$A$, that this bipartition is unique.
    This proves the claim.
    
    To prove the second part of the lemma, first observe that any path between a leaf~$x\in X'-\{z\}$ and~$z$ in the network~$N'$ corresponds to a path between~$x$ and~$q$ in~$N$.
    Now the multiset of distances between~$x$ and~$q$ in~$N$ can be obtained by finding the multiset of distances between~$x$ and~$a_1$ that used the edges~$qp_1,p_1a_1$, and subtracting 2 from each element.
    This is precisely the set~$A-2$ that we have found above.
    For any other leaf~$y\in X'-\{z\}$, we have that all paths between~$x$ and~$y$ are unaffected by the replacement of the blob by~$z$, as the blob is pendant in~$N$.
    Therefore~$d(x,y)$ remains unchanged for~$x,y\in X'-\{z\}$.
\end{proof}

It is again easy to reconstruct the blob after reconstructing the reduced network, since we know that $(a_1,\ldots ,a_k)$ must form a chain on the blob, in that order.

\subsubsection{Pendant level-\emph{2} blobs}

We adopt the following notation for pendant level-$2$ blobs.
Let~$B$ be a pendant level-$2$ blob, and let~$a,b,c,d$ denote the four chains contained in~$B$ of lengths~$k,\ell,m,n \geq0$ such that chains~$c$ and~$d$ are on the same side of $B$ as the non-trivial cut-edge.
Then we say that~$B$ is of the form~$(k,\ell,m,n)$.
For ease of notation, a side without leaves is seen as a length-0 chain.
See Figure~\ref{fig:L2NetPendantL2BlobChain} for pendant level-$2$ blobs of the forms~$(k,0,0,0)$ and~$(k,\ell,0,0)$.

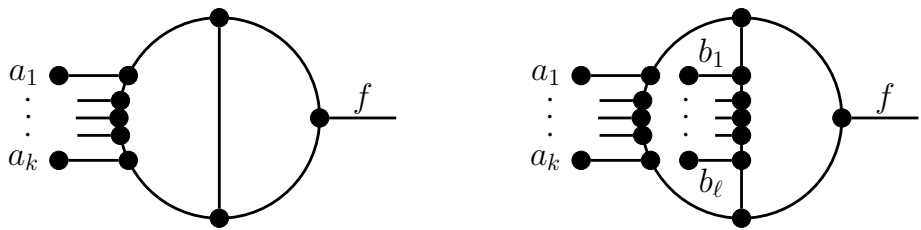
\begin{figure}[ht]
    \begin{subfigure}{.5\textwidth}
    \centering
    \resizebox{0.8\textwidth}{!}{
        \begin{tikzpicture}[every node/.style={draw, circle, fill, inner sep=0pt}]
            \tikzset{edge/.style={very thick}}
            \draw[black, very thick] (0,0) circle (1.5);
            
            \node[] (north) at (90:1.5)  {a};
            \node[] (south) at (270:1.5) {a};
            \node[] (east)  at (0:1.5)   {a};
            
            \node[] (p_1) at (155:1.5) {a};
            \node[] (p_2) at (170:1.5) {a};
            \node[] (p_3) at (180:1.5) {a};
            \node[] (p_4) at (190:1.5) {a};
            \node[] (p_k) at (205:1.5) {a};
            \node[draw = none, fill = none,right= 1cm of east] (p) {}; 
         	
         	\node[left= 0.75cm of p_1] (a_1)    {a};
         	\node[draw = none, fill = none, left=0.5cm of p_2] (a_2) {};
         	\node[draw = none, fill = none, left=0.5cm of p_3] (a_3) {};
         	\node[draw = none, fill = none, left=0.5cm of p_4] (a_4) {};
         	\node[left= 0.75cm of p_k] (a_k)    {a};
         	
            \node[draw = none, fill = none, left=1mm of a_1]  (la_1) {\large $a_1$};
            \node[draw = none, fill = none, left=0.55cm of a_2]  (la_2) {\large $\cdot$};
            \node[draw = none, fill = none, left=0.52cm of a_3]  (la_3) {\large $\cdot$};
            \node[draw = none, fill = none, left=0.55cm of a_4]  (la_4) {\large $\cdot$};
            \node[draw = none, fill = none, left=1mm of a_k]  (la_k) {\large $a_k$};
         	
         	\draw[edge] (north) -- (south);
            \draw[edge] (east) -- (p) node[draw = none, fill = none, above=0.01mm, midway]{\large $f$};
         	\draw[edge] (p_1) -- (a_1);
         	\draw[edge] (p_2) -- (a_2);
         	\draw[edge] (p_3) -- (a_3);
         	\draw[edge] (p_4) -- (a_4);
         	\draw[edge] (p_k) -- (a_k);
        \end{tikzpicture}
        }
    \end{subfigure}
    \begin{subfigure}{.5\textwidth}
        \centering
        \resizebox{0.8\textwidth}{!}{
        \begin{tikzpicture}[every node/.style={draw, circle, fill, inner sep=0pt}]
         	\tikzset{edge/.style={very thick}}

            \draw[black, very thick] (0,0) circle (1.5);

            \node[] (north) at (90:1.5)  {a};
            \node[] (south) at (270:1.5) {a};
         	\node[] (east) at (0:1.5) {a};
         	
            \node[] (p_1) at (155:1.5) {a};
            \node[] (p_2) at (170:1.5) {a};
            \node[] (p_3) at (180:1.5) {a};
            \node[] (p_4) at (190:1.5) {a};
            \node[] (p_k) at (205:1.5) {a};
         	
         	\node[] (q_1) at (0,{1.5*sin(155)}) {a};
         	\node[] (q_2) at (0,{1.5*sin(170)}) {a};
         	\node[] (q_3) at (0,{1.5*sin(180)}) {a};
         	\node[] (q_4) at (0,{1.5*sin(190)}) {a};
         	\node[] (q_l) at (0,{1.5*sin(205)}) {a};
         	
            \node[draw = none, fill = none,right= 1cm of east] (p) {};
         	
         	\node[left=0.75 of p_1] (a_1) {a};
         	\node[draw = none, fill = none, left = 0.5 of p_2] (a_2) {};
         	\node[draw = none, fill = none, left = 0.5 of p_3] (a_3) {};
         	\node[draw = none, fill = none, left = 0.5 of p_4] (a_4) {};
         	\node[left=0.75 of p_k] (a_k) {a};
         	
         	\node[left=0.5 of q_1] (b_1) {a};
         	\node[draw = none, fill = none, left = 0.25 of q_2] (b_2) {};
         	\node[draw = none, fill = none, left = 0.25 of q_3] (b_3) {};
         	\node[draw = none, fill = none, left = 0.25 of q_4] (b_4) {};
         	\node[left=0.5 of q_l] (b_l) {a};
         	
            \node[draw = none, fill = none, left=1mm of a_1]  (la_1) {\large $a_1$};
            \node[draw = none, fill = none, left=0.55cm of a_2]  (la_2) {\large $\cdot$};
            \node[draw = none, fill = none, left=0.53cm of a_3]  (la_3) {\large $\cdot$};
            \node[draw = none, fill = none, left=0.55cm of a_4]  (la_4) {\large $\cdot$};
            \node[draw = none, fill = none, left=1mm of a_k]  (la_k) {\large $a_k$};
            \node[draw = none, fill = none, above right=0.3mm and 0.3mm of b_1]  (lb_1) {\large $b_1$};
            \node[draw = none, fill = none, left=0.25cm of b_2]  (lb_2) {\large $\cdot$};
            \node[draw = none, fill = none, left=0.25cm of b_3]  (lb_3) {\large $\cdot$};
            \node[draw = none, fill = none, left=0.25cm of b_4]  (lb_4) {\large $\cdot$};
            \node[draw = none, fill = none, below right=0.3mm and 0.3mm of b_l]  (lb_l) {\large $b_\ell$};
         	
            \draw[edge] (east) -- (p) node[draw = none, fill = none, above=0.01mm, midway]{\large $f$};
         	\draw[edge]         (north) -- (south);
         	\draw[edge]         (p_1) -- (a_1);
         	\draw[edge]         (p_2) -- (a_2);
         	\draw[edge]         (p_3) -- (a_3);
         	\draw[edge]         (p_4) -- (a_4);
         	\draw[edge]         (p_k) -- (a_k);
         	\draw[edge]         (q_1) -- (b_1);
         	\draw[edge]         (q_2) -- (b_2);
         	\draw[edge]         (q_3) -- (b_3);
         	\draw[edge]         (q_4) -- (b_4);
         	\draw[edge]         (q_l) -- (b_l);
        \end{tikzpicture}
        }
    \end{subfigure}
    \caption{A pendant level-$2$ blob of the form~$(k,0,0,0)$ containing the chain~$(a_1,\dots,a_k)$ (left) and a pendant level-$2$ blob of the form~$(k,\ell,0,0)$ containing the chains~$(a_1,\ldots,a_k)$ and~$(b_1,\ldots,b_\ell)$.
    The edges labelled~$f$ denote the non-trivial cut-edges in both networks.}
    \label{fig:L2NetPendantL2BlobChain}
\end{figure}

\begin{lemma}\label{lem:L2NetPendantL2BlobChain}
    A level-$2$ network~$N$ contains a pendant level-$2$ blob of the form~$(k,0,0,0)$ for~$k\geq2$ with the chain~$(a_1,\ldots,a_k)$
    if and only if~$d(a_1,a_k) = \{5^1,6^1,(k+1)^1\}$.
\end{lemma}
\begin{proof}
    Suppose first that~$N$ contains a pendant level-$2$ blob~$B$ of the form $(k,0,0,0)$.
    Let~$e$ denote the non-trivial cut-edge that is incident to~$B$.
    Then the path from~$a_1$ to~$a_k$ that uses the side of~$B$ without~$e$ and without the chain, the side of~$B$ with~$e$, and the side of~$B$ with the chain are of distances~$5,6$, and~$k+1$ respectively.
    \medskip
    
    Suppose now that there exists a chain~$(a_1,\ldots,a_k)$ where~$k\geq2$ such that~$d(a_1,a_k)=\{5^1,6^1,(k+1)^1\}$.
    First, since~$|d(a_1,a_k)|>2$, we note that the chain~$(a_1,\ldots,a_k)$ must be contained in a level-$2$ blob.
    Consider a level-$2$ blob~$B$ that contains the chain~$(a_1,\ldots,a_k)$ on one of its sides, and suppose that there is a single non-trivial cut-edge~$e$ on another one of its sides.
    There must be at least one such edge~$e$ because otherwise there would be parallel edges.
    Currently we have that~$d(a_1,a_k) = \{5^1,6^1,(k+1)^1\}$: adding more cut-edges (trivial or non-trivial) to the sides of~$B$ would change the set of distances.
    Since~$B$ is incident to exactly one non-trivial cut-edge, it is a level-$2$ pendant blob.
\end{proof}

\begin{lemma}\label{lem:L2NetPendantL2BlobLeaf}
    A level-$2$ network~$N$ contains a pendant level-$2$ blob of the form~$(1,0,0,0)$ containing the leaf~$a$
    if and only if~$d_m(a,x)\geq6$ for all~$x\in X-\{a\}$ and for any two leaves~$y,z\in X-\{a\}$,~$d_m(a,y)+d_m(a,z) - d_m(y,z)\geq8$.
\end{lemma}
\begin{proof}
    Suppose first that a pendant level-$2$ blob~$B$ contains only the leaf~$a$.
    Let~$uv$ denote the non-trivial cut-edge incident to~$B$, where~$u$ is the vertex that is on~$B$.
    Now, the shortest distance from~$a$ to~$u$ is exactly~$3$.
    Furthermore, the shortest distance from~$u$ to a leaf~$x$ that is not~$a$ is at least~$3$, since such a path must contain the edge~$uv$, an edge of another blob, and an edge incident to~$x$.
    In particular, such a path must contain an edge of another blob since all leaves are assumed to be contained in blobs.
    Therefore~$d_m(a,x)\geq6$ for all~$x\in X-\{a\}$.
    To prove the second statement, let~$y,z\in X-\{a\}$.
    Then by the triangle inequality, we have
    \[d_m(a,y) + d_m(a,z) - d_m(y,z) = d_m(v,y) + d_m(v,z) - d_m(y,z) + 8 \ge 8.\]
    
    Now suppose that~$d_m(a,x)\geq6$ for all~$x\in X-\{a\}$ and for any two leaves~$y,z\in X-\{a\}$, we have~$d_m(a,y)+d_m(a,z) - d_m(y,z)\geq8$.
    The first condition implies that~$(a)$ is a maximal chain.
    Suppose first that~$a$ was contained in a level-$1$ blob~$B$.
    Note that~$B$ cannot be pendant as otherwise the network would have parallel edges.
    Let~$p_a$ denote the neighbour of~$a$ (a vertex of~$B$), and let~$p_y,p_z$ denote the two neighbours of~$p_a$ on~$B$ that are not~$a$.
    The vertices~$p_y$ and~$p_z$ are necessarily incident to non-trivial cut-edges, as otherwise~$a$ would be contained in a chain, in which case the condition~$d_m(a,x)\geq6$ would be violated for some leaf~$x$ in the chain.
    Now let~$y$ and~$z$ denote any leaves in~$X-\{a\}$ that can be reached from~$B$ via the cut-edges incident to~$p_y$ and~$p_z$ respectively.
    Then we have that~$d_m(a,y)+d_m(a,z)-d_m(y,z) = 2$ if a shortest path between~$p_y$ and~$p_z$ passes the vertex~$p_a$, and we have~$d_m(a,y)+d_m(a,z)-d_m(y,z) = 3$ otherwise.
    This contradicts our second condition, and therefore we may assume that the leaf~$a$ is contained in a level-$2$ blob~$B$.
    Suppose that~$B$ is a non-pendant blob, in other words, that there are at least two non-trivial cut-edges incident to~$B$.
    Take two non-trivial cut-edges that are closest to~$a$, and take any two leaves~$y$ and~$z$ that can be reached from~$B$ via these cut-edges.
    The shortest distance from~$a$ to the endpoints of these cut-edges on~$B$ is at most~$3$.
    Therefore we have~$d_m(a,y) + d_m(a,z) - d_m(y,z) \leq 6$, which contradicts our second condition.
    Therefore we may assume that the leaf~$a$ is contained in a pendant level-$2$ blob~$B$.
    But aside from the leaf~$a$ and the single non-trivial cut-edge, no other cut-edges can be incident to~$B$.
    Indeed, having another leaf that is contained in~$B$ violates the first condition, and having another non-trivial cut-edge contradicts the fact that~$B$ was pendant.
    Therefore~$B$ is a pendant level-$2$ blob of the form~$(1,0,0,0)$ that contains a single leaf~$a$.
\end{proof}

\begin{lemma}\label{lem:L2NetPendantL2BlobChainAdjust}
    Let~$N$ be a level-$2$ network on~$X$ containing a pendant level-$2$ blob of the form~$(k,0,0,0)$ for~$k\geq1$ with the chain~$(a_1,\ldots,a_k)$.
    Then we can replace the pendant blob by a leaf~$z$ to obtain a network~$N'$ on~$X' = X\cup\{z\}-\{a_1,\ldots,a_k\}$.
    For every~$x\in X'-\{z\}$, we can uniquely partition the multiset of distances~$d(x,a_1)$ into four equal sized sets~$A,B,C,D$ such that~$A - 3 = B-4 = C-(k+2) = D-(k+3)$. Then the multisets of distances of~$N'$ contains the elements
    \[d^{N'}(x,y) =
    \begin{cases} 
        d^N(x,y)    & \text{if }x,y\in X'-\{z\} \\
        A-3         & \text{if }y=z.
    \end{cases}
    \]
\end{lemma}
\begin{proof}
    We first show that the partition of~$d(x,a_1)$ exists and that it is unique.
    Let~$B$ denote the pendant level-$2$ blob containing~$(a_1,\ldots,a_k)$, and let~$q$ denote the vertex in~$B$ that is an endpoint of a non-trivial cut-edge.
    Let~$x\in X'-\{z\}$.
    Every path from~$x$ to~$a_1$ consists of a path from~$x$ to~$q$ and a path from~$q$ to~$a_1$. There are four possible paths from~$q$ to~$a_1$ of lengths~$3,4,k+2$, and~$k+3$.
    By an analogous argument used in the proof of Lemma~\ref{lem:L2NetPendantL1BlobAdjust}, there is a unique partition of~$d(x,a_1)$ into four equal sized sets~$A,B,C,D$ such that~$A - 3 = B-4 = C-(k+2) = D-(k+3)$.
    
    Upon replacing the pendant blob~$B$ by a leaf~$z$, we note that the multiset of distances between a leaf~$x\in X'-\{z\}$ and~$z$ in~$N'$ is equivalent to the multiset of distances between~$x$ and~$q$ in~$N$.
    This multiset of distances is precisely the set~$A-3$.
    Let~$y\in X'-\{z\}$ be another leaf that is not~$x$.
    Then all paths between~$x$ and~$y$ in~$N$ are unaffected after replacing~$B$ by a leaf~$z$; therefore~$d^{N'}(x,y) = d^{N}(x,y)$.
\end{proof}

\paragraph{Pendant level-$2$ blobs with at least two chains}

\begin{lemma}\label{lem:L2NetPendantL2BlobTwoChains}
    A level-$2$ network $N$ on~$X$ contains a pendant level-$2$ blob of the form~$(k,\ell,0,0)$ with chains~$a = (a_1,\ldots,a_k)$ and~$b=(b_1,\ldots,b_\ell)$ with~$k,\ell\ge1$
    if and only if~$a$ and~$b$ are adjacent twice, and for all~$c\in a\cup b$, we have~$d_m(c,x)\geq6$ for all~$x\in X-(a\cup b)$ and~$d_m(c,y)+d_m(c,z) - d_m(y,z)\geq8$ for any two leaves~$y,z\in X-(a\cup b)$.
\end{lemma}
\begin{proof}
    One direction follows an analogous argument used in the proof of Lemma~\ref{lem:L2NetPendantL2BlobLeaf}.
    \medskip
    
    To show the other direction, suppose that~$a$ and~$b$ are adjacent twice, and for all~$c\in a\cup b$, we have~$d_m(c,x)\geq6$ for all~$x\in X-(a\cup b)$ and~$d_m(c,y)+d_m(c,z) - d_m(y,z)\geq8$ for any two leaves~$y,z\in X-(a\cup b)$.
    Since~$a$ and~$b$ are adjacent twice, either~$a$ and~$b$ are contained in the same level-$1$ blob such that the cycle of the blob is~$up_1p_2\ldots p_kvq_1q_2\ldots q_\ell u$ where~$p_i$ and~$q_j$ denote the neighbours of~$a_i$ and~$b_j$ for~$i\in[k], j\in[\ell]$, respectively, and~$u$ and~$v$ are incident to non-trivial cut-edges,
    or~$a$ and~$b$ are contained in the same level-$2$ blob~$B$ in which~$a$ and~$b$ are on two different sides of~$B$ and there are no other vertices that subdivide these two sides of~$B$ (see Figure~\ref{fig:L2NetPendantL2BlobTwoChains}).
    
    In the first case, let~$B$ denote the level-$1$ blob.
    We take leaves~$y$ and~$z$ that can be reached from~$B$ via the two non-trivial cut-edges.
    Without loss of generality, assume that~$k\leq \ell$. Then the shortest path from~$y$ to~$z$ must pass through the neighbours of~$a_i$ for all~$i\in[k]$.
    But then for any~$c\in a$, we have that
    \[d_m(c,y) + d_m(c,z) - d_m(y,z) = 2,\]
    which contradicts our original assumption.
    
    In the second case, let~$B$ denote the level-$2$ blob and let~$e$ denote the side of~$B$ that does not contain~$a$ nor~$b$.
    Since the network contains at least two blobs, the side~$e$ must be incident to at least one non-trivial cut-edge.
    Suppose for a contradiction that there are at least two cut-edges incident to the side~$e$.
    Let~$p$ and~$q$ denote the vertices on side~$e$ such that if~$k\ge2$ then they have shortest distance~$3$ and~$4$ from~$a_1$, respectively, and if~$k=1$ then they have shortest distance~$3$ and at most~$4$ from~$a_1$, respectively.
    Note first that the cut-edges incident to~$p$ and~$q$ must be non-trivial cut-edges -- otherwise this would contradict our assumption that for any leaf~$x\in X-(a\cup b)$, we have~$d_m(a_1,x)\geq 6$.
    Let~$y$ and~$z$ denote leaves that can be reached from~$B$ via the cut-edges incident to~$p$ and~$q$, respectively.
    Then
    \begin{align*}
        d_m(a_1,y) + d_m(a_1,z) - d_m(y,z) &\leq 3 + d_m(p,y) + 4 + d_m(q,z) - d_m(y,z)\\
        &= 7 - d_m(p,q)\\
        &\le 6,
    \end{align*}
    where the final inequality follows as~$d_m(p,q)>0$.
    This is a contradiction.
    Therefore there is exactly one cut-edge that is incident to the side~$e$, from which it follows that~$a$ and~$b$ are the only chains contained in a pendant level-$2$ blob of the form~$(k,\ell,0,0)$.
\end{proof}

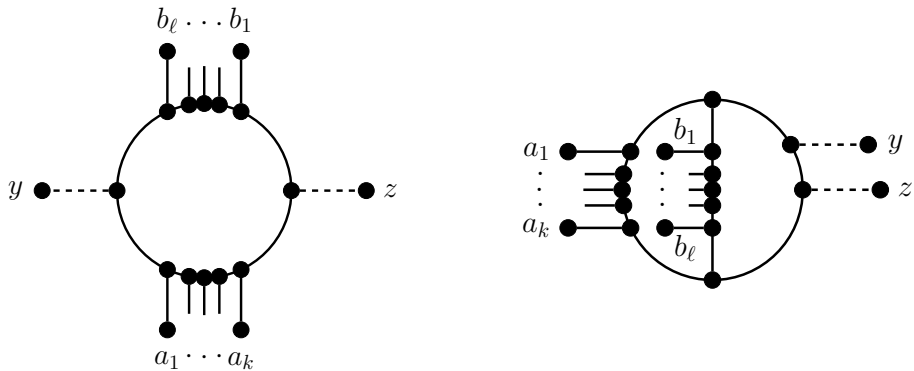
\begin{figure}[ht]
    \begin{subfigure}{.5\textwidth}
        \centering
        \resizebox{0.8\textwidth}{!}{
        \begin{tikzpicture}[every node/.style={draw, circle, fill, inner sep=0pt}]
         	\tikzset{edge/.style={very thick}}
            \draw[black, very thick] (0,0) circle (1.5);

            \node[] (west)  at (180:1.5)    {a};
            \node[] (east)  at (0:1.5)      {a};
            
            \node[] (p_1)   at (245:1.5)    {a};
            \node[] (p_2)   at (260:1.5)    {a};
            \node[] (p_3)   at (270:1.5)    {a};
            \node[] (p_4)   at (280:1.5)    {a};
            \node[] (p_k)   at (295:1.5)    {a};
            
            \node[] (q_1)   at (65:1.5)     {a};
            \node[] (q_2)   at (80:1.5)     {a};
            \node[] (q_3)   at (90:1.5)     {a};
            \node[] (q_4)   at (100:1.5)    {a};
            \node[] (q_l)   at (115:1.5)    {a};
            
            \node[left= 1 of west] (y) {a};
            \node[right=1 of east] (z) {a};
            \node[below=0.75 of p_1] (a_1) {a};
            \node[draw = none, fill = none, below=0.5 of p_2] (a_2) {};
            \node[draw = none, fill = none, below=0.5 of p_3] (a_3) {};
            \node[draw = none, fill = none, below=0.5 of p_4] (a_4) {};
            \node[below=0.75 of p_k] (a_k) {a};
            \node[above=0.75 of q_1] (b_1) {a};
            \node[draw = none, fill = none, above=0.5 of q_2] (b_2) {};
            \node[draw = none, fill = none, above=0.5 of q_3] (b_3) {};
            \node[draw = none, fill = none, above=0.5 of q_4] (b_4) {};
            \node[above=0.75 of q_l] (b_l) {a};

            \node[draw = none, fill = none, below=1mm of a_1]  (la_1)   {\large $a_1$};
            \node[draw = none, fill = none, below=0.55cm of a_2]  (la_2) {\large $\cdot$};
            \node[draw = none, fill = none, below=0.54cm of a_3]  (la_3) {\large $\cdot$};
            \node[draw = none, fill = none, below=0.55cm of a_4]  (la_4) {\large $\cdot$};
            \node[draw = none, fill = none, below=1mm of a_k]  (la_k)   {\large $a_k$};
            \node[draw = none, fill = none, above=1mm of b_1]  (lb_1)   {\large $b_1$};
            \node[draw = none, fill = none, above=0.55cm of b_2]  (lb_2) {\large $\cdot$};
            \node[draw = none, fill = none, above=0.54cm of b_3]  (lb_3) {\large $\cdot$};
            \node[draw = none, fill = none, above=0.55cm of b_4]  (lb_4) {\large $\cdot$};
            \node[draw = none, fill = none, above=1mm of b_l]  (lb_l)   {\large $b_\ell$};
            \node[draw = none, fill = none, left=1mm of y]     (ly)     {\large $y$};
            \node[draw = none, fill = none, right=1mm of z]    (lz)     {\large $z$};
            
            \draw[edge, dashed] (west) -- (y);
            \draw[edge, dashed] (east) -- (z);
            \draw[edge] (p_1) -- (a_1);
            \draw[edge] (p_2) -- (a_2);
            \draw[edge] (p_3) -- (a_3);
            \draw[edge] (p_4) -- (a_4);
            \draw[edge] (p_k) -- (a_k);
            \draw[edge] (q_1) -- (b_1);
            \draw[edge] (q_2) -- (b_2);
            \draw[edge] (q_3) -- (b_3);
            \draw[edge] (q_4) -- (b_4);
            \draw[edge] (q_l) -- (b_l);
            \end{tikzpicture}
            }
    \end{subfigure}\hfill
    \begin{subfigure}{.5\textwidth}
        \centering
        \resizebox{0.8\textwidth}{!}{
        \begin{tikzpicture}[every node/.style={draw, circle, fill, inner sep=0pt}]
         	\tikzset{edge/.style={very thick}}

            \draw[black, very thick] (0,0) circle (1.5);

            \node[] (north) at (90:1.5)  {a};
            \node[] (south) at (270:1.5) {a};
         	
         	\node[] (py) at (30:1.5)  {a};
         	\node[] (east) at (0:1.5) {a};
         	
            \node[] (p_1) at (155:1.5) {a};
            \node[] (p_2) at (170:1.5) {a};
            \node[] (p_3) at (180:1.5) {a};
            \node[] (p_4) at (190:1.5) {a};
            \node[] (p_k) at (205:1.5) {a};
         	
         	\node[] (q_1) at (0,{1.5*sin(155)}) {a};
         	\node[] (q_2) at (0,{1.5*sin(170)}) {a};
         	\node[] (q_3) at (0,{1.5*sin(180)}) {a};
         	\node[] (q_4) at (0,{1.5*sin(190)}) {a};
         	\node[] (q_l) at (0,{1.5*sin(205)}) {a};
         	
         	\node[right=1 of py]    (y) {a};
         	\node[right=1 of east]  (z) {a};
         	
         	\node[left=0.75 of p_1] (a_1) {a};
         	\node[draw = none, fill = none, left = 0.5 of p_2] (a_2) {};
         	\node[draw = none, fill = none, left = 0.5 of p_3] (a_3) {};
         	\node[draw = none, fill = none, left = 0.5 of p_4] (a_4) {};
         	\node[left=0.75 of p_k] (a_k) {a};
         	
         	\node[left=0.5 of q_1] (b_1) {a};
         	\node[draw = none, fill = none, left = 0.25 of q_2] (b_2) {};
         	\node[draw = none, fill = none, left = 0.25 of q_3] (b_3) {};
         	\node[draw = none, fill = none, left = 0.25 of q_4] (b_4) {};
         	\node[left=0.5 of q_l] (b_l) {a};
         	
            \node[draw = none, fill = none, left=1mm of a_1]  (la_1) {\large $a_1$};
            \node[draw = none, fill = none, left=0.55cm of a_2]  (la_2) {\large $\cdot$};
            \node[draw = none, fill = none, left=0.54cm of a_3]  (la_3) {\large $\cdot$};
            \node[draw = none, fill = none, left=0.55cm of a_4]  (la_4) {\large $\cdot$};
            \node[draw = none, fill = none, left=1mm of a_k]  (la_k) {\large $a_k$};
            \node[draw = none, fill = none, above right=0.3mm and 0.3mm of b_1]  (lb_1) {\large $b_1$};
            \node[draw = none, fill = none, left=0.25cm of b_2]  (lb_2) {\large $\cdot$};
            \node[draw = none, fill = none, left=0.25cm of b_3]  (lb_3) {\large $\cdot$};
            \node[draw = none, fill = none, left=0.25cm of b_4]  (lb_4) {\large $\cdot$};
            \node[draw = none, fill = none, below right=0.3mm and 0.3mm of b_l]  (lb_l) {\large $b_\ell$};
            \node[draw = none, fill = none, right=1mm of y]  (ly) {\large $y$};
            \node[draw = none, fill = none, right=1mm of z]  (lz) {\large $z$};
         	
         	\draw[edge, dashed] (py) -- (y);
         	\draw[edge, dashed] (east) -- (z);
         	\draw[edge]         (north) -- (south);
         	\draw[edge]         (p_1) -- (a_1);
         	\draw[edge]         (p_2) -- (a_2);
         	\draw[edge]         (p_3) -- (a_3);
         	\draw[edge]         (p_4) -- (a_4);
         	\draw[edge]         (p_k) -- (a_k);
         	\draw[edge]         (q_1) -- (b_1);
         	\draw[edge]         (q_2) -- (b_2);
         	\draw[edge]         (q_3) -- (b_3);
         	\draw[edge]         (q_4) -- (b_4);
         	\draw[edge]         (q_l) -- (b_l);
        \end{tikzpicture}
        }
    \end{subfigure}
    \caption{The two possibilities for when two chains~$a = (a_1,\ldots,a_k)$ and~$b = (b_1,\ldots,b_\ell)$ are adjacent twice and they are not contained in a pendant level-$2$ blob, as in the proof of Lemma~\ref{lem:L2NetPendantL2BlobTwoChains}. A level-$1$ blob (left) and a non-pendant level-$2$ blob (right).
    The dashed edges in both networks represent paths that are not trivial cut-edges from the blob to the leaves~$y$ and~$z$.
    In the non-pendant level-$2$ blob, there could be additional cut-edges on the side not containing the chains~$a$ and~$b$.}
    \label{fig:L2NetPendantL2BlobTwoChains}
\end{figure}

\begin{lemma}\label{lem:L2NetPendantL2BlobTwoChainsAdjust}
    Let~$N$ be a level-$2$ network on~$X$ that contains a pendant level-$2$ blob of the form~$(k,\ell,0,0)$ with chains~$a= (a_1,\ldots,a_k)$ and~$b = (b_1,\ldots,b_\ell)$.
    Then we can replace the pendant blob by a leaf~$z$ to obtain a network~$N'$ on~$X' = X\cup\{z\}-(a\cup b)$.
    For every~$x\in X'$, we can uniquely partition the multiset of distances~$d(x,a_1)$ into four equal sized sets~$A,B,C,D$ such that~$A - 3 = B-(\ell+4) = C-(k+2) = D-(k+\ell+3)$. Then the multisets of distances of~$N'$ contains the elements
    \[d^{N'}(x,y) =
    \begin{cases} 
        d^N(x,y)    & \text{if }x,y\in X'-\{z\} \\
        A-3         & \text{if }y=z.
    \end{cases}
    \]
\end{lemma}
\begin{proof}
    The proof is analogous to that of Lemma~\ref{lem:L2NetPendantL2BlobChainAdjust}.
\end{proof}
\medskip

\paragraph{Chain-Adjacency Graphs}

We have now dealt with pendant level-$2$ blobs of the forms~$(k,0,0,0)$ (Lemmas~\ref{lem:L2NetPendantL2BlobChain} and~\ref{lem:L2NetPendantL2BlobLeaf}) and~$(k,\ell,0,0)$ (Lemma~\ref{lem:L2NetPendantL2BlobTwoChains}).
For the remaining four cases (ignoring symmetric cases) left to examine, $(k,0,m,0);$ $(k,0,m,n); (k,\ell,m,0);$ and~$(k,\ell,m,n)$, we employ the following graph.

\begin{definition}
    A \emph{chain-adjacency graph} (CAG) has a vertex for each chain, and between two vertices,
    \begin{itemize}
        \item we insert a red edge if the chains are adjacent once and two red edges if the chains are adjacent twice; and
        \item if the two chains are adjacent once, we insert a green edge for each length-$5$ path between endpoints of the chains (one per chain) that does not contain any edges of the two chains.
    \end{itemize}
\end{definition}

The condition for joining two vertices on the CAG via a green edge can indeed be verified from the multisets of distances.
Let~$a=(a_1,\ldots,a_k)$ and~$b=(b_1,\ldots,b_\ell)$ denote two chains that are adjacent once, and suppose without loss of generality that~$d_m(a_1,b_1) = 4$.
To count the number of green edges between~$a$ and~$b$, we fall into the~$9$ cases shown in Table~\ref{tab:GreenEdge}.
This number is obtained by taking the multiplicity of~$5$'s in the multiset of distances between a pair of endpoints, minus the number of length-$5$ paths that pass through edges of the chains.
Let~$(A,m_A) = d(a_1,b_1)$;~$(B,m_B) = d(a_1,b_\ell)$;~$(C,m_C) = d(a_k,b_1)$;~$(D,m_D) = d(a_k,b_\ell)$.


\begin{table}[ht]
    \centering
    \begin{tabular}{l|l|l|l|l}
    \cline{2-4}
                                & $\ell=1$                                  & $\ell=2$                                         & $\ell>2$                                            &  \\ \cline{1-4}
    \multicolumn{1}{|l|}{$k=1$} & $m_A(5)$                   & $m_{A+B}(5)-1$        & $m_{A+B}(5)$          &  \\ \cline{1-4}
    \multicolumn{1}{|l|}{$k=2$} & $m_{A+C}(5)-1$ & $m_{A+B+C+D}(5)-2$ & $m_{A+B+C+D}(5)-1$ &  \\ \cline{1-4}
    \multicolumn{1}{|l|}{$k>2$} & $m_{A+C}(5)$   & $m_{A+B+C+D}(5)-1$ & $m_{A+B+C+D}(5)$  &  \\ \cline{1-4}
    \end{tabular}
    \caption{The number of green edges between two adjacent chains~$a = (a_1,\ldots,a_k)$ and~$b = (b_1,\ldots, b_\ell)$ for different~$k$ and~$\ell$ values.}
    \label{tab:GreenEdge}
\end{table}

We only insert green edges between chains that are adjacent, rather than between all chains that are distance-$5$ apart, to ensure that chains contained in different blobs are not connected in the CAG.
Since we may assume that all leaves are contained in blobs, we note that two chains are adjacent and in the same blob if and only if they are connected by a red edge in the CAG. 
Note that there may be multiple edges between two vertices in a CAG (see Figure~\ref{fig:CAG}).
We now show how we can use the CAG to distinguish the configurations of pendant blobs from non-pendant blobs, and how it can be used to distinguish the remaining level-$2$ pendant blob structures.

Observe that every edge in the CAG corresponds to a distinct distance-$4$ or distance-$5$ path between a pair of chain endpoints.
We say that this path in the network is \emph{covered} by the edge of the CAG.
In particular, we also say that the edges of the path of the network is \emph{covered} by this edge of the CAG.
Note that an edge of a network can be covered by more than one edge of the CAG.
See Figure~\ref{fig:CAG} (c) for an example of a distance-$5$ path that is covered by an edge in the CAG.

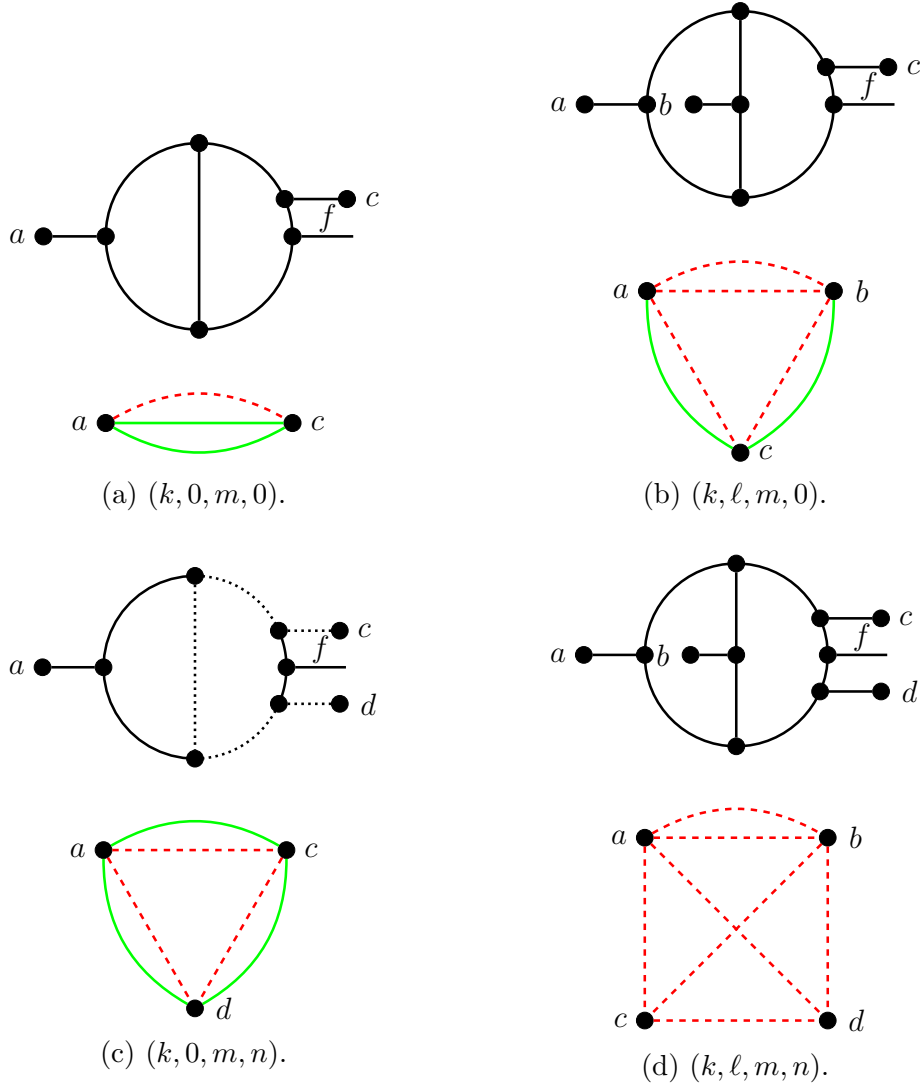
\begin{figure}
    \begin{subfigure}[t]{.475\textwidth}
        \centering
        \resizebox{0.8\textwidth}{!}{
        \begin{tikzpicture}[every node/.style={draw, circle, fill, inner sep=0pt}]
         	\tikzset{edge/.style={very thick}}
         	\begin{scope}[xshift=0cm,yshift=0cm]
                \draw[black, very thick] (0,0) circle (1.5);
                \node[] (u1)    at (0,1.5)          {a};
                \node[] (u2)    at (0,-1.5)         {a};
                \node[] (u3)    at (-1.5,0)         {a};
                \node[] (u4)    at (1.5,0)          {a};
                \node[] (u5)    at (1.3747,0.6)    {a};
                
                \node[] (a)     at (-2.5,0)         {a};
                \node[] (c)     at (2.3747,0.6)    {a};
                
                \node[draw = none, fill = none] (ce)    at (2.5,0)   {};
                \node[draw = none, fill = none, left=1mm of a]  (la) {\large $a$};
                \node[draw = none, fill = none, right=1mm of c] (lc) {\large $c$};
                
                \draw[edge] (u1) -- (u2);
                \draw[edge] (u4) -- (ce) node[draw = none, fill = none, above=0.01mm, midway]{\large $f$};
                \draw[edge] (u3) -- (a);
                \draw[edge] (u5) -- (c);
            \end{scope}
            \begin{scope}[xshift=0cm,yshift=-3cm]
                \node[] (a)     at (-1.5,0)         {a};
                \node[] (c)     at (1.5,0)          {a};
                \draw[edge, green]                  (a) -- (c);
                \draw[edge, bend right, green]      (a) edge (c);
                \draw[edge, bend left, red, dashed] (a) edge (c);
                
                \node[draw = none, fill = none, left=1mm of a]  (la) {\large $a$};
                \node[draw = none, fill = none, right=1mm of c] (lc) {\large $c$};
            \end{scope}
        \end{tikzpicture}
        }
        \caption{$(k,0,m,0)$.}
    \end{subfigure}\hfill
    \begin{subfigure}[t]{.475\textwidth}
        \centering
        \resizebox{0.8\textwidth}{!}{
        \begin{tikzpicture}[every node/.style={draw, circle, fill, inner sep=0pt}]
         	\tikzset{edge/.style={very thick}}
         	\begin{scope}[xshift=0cm,yshift=0cm]
                \draw[black, very thick] (0,0) circle (1.5);
                \node[] (u1)    at (0,1.5)          {a};
                \node[] (u2)    at (0,-1.5)         {a};
                \node[] (u3)    at (-1.5,0)         {a};
                \node[] (u4)    at (1.5,0)          {a};
                \node[] (u5)    at (1.3747,0.6)     {a};
                \node[] (u6)    at (0,0)            {a};
                
                \node[] (a)     at (-2.5,0)         {a};
                \node[] (b)     at (-0.75,0)        {a};
                \node[] (c)     at (2.3747,0.6)     {a};
                
                \node[draw = none, fill = none] (ce)    at (2.5,0)   {};
                \node[draw = none, fill = none, left=1mm of a]  (la) {\large $a$};
                \node[draw = none, fill = none, left=1mm of b]  (lb) {\large $b$};
                \node[draw = none, fill = none, right=1mm of c] (lc) {\large $c$};
                
                \draw[edge] (u1) -- (u2);
                \draw[edge] (u4) -- (ce) node[draw = none, fill = none, above=0.01mm, midway]{\large $f$};
                \draw[edge] (u3) -- (a);
                \draw[edge] (u5) -- (c);
                \draw[edge] (u6) -- (b);
            \end{scope}
            \begin{scope}[xshift=0cm,yshift=-3cm]
                \node[] (a)     at (-1.5,0)             {a};
                \node[] (b)     at (1.5,0)              {a};
                \node[] (c)     at (0,-2.598)           {a};
                \draw[edge, red, dashed]                (a) -- (b);
                \draw[edge, red, dashed]                (a) -- (c);
                \draw[edge, red, dashed]                (b) -- (c);
                \draw[edge, bend left, red, dashed]     (a) edge (b);
                \draw[edge, bend right, green]          (a) edge (c);
                \draw[edge, bend left, green]           (b) edge (c);

                \node[draw = none, fill = none, left=1mm of a]  (la) {\large $a$};
                \node[draw = none, fill = none, right=1mm of b]  (lb) {\large $b$};
                \node[draw = none, fill = none, right=1mm of c] (lc) {\large $c$};
            \end{scope}
        \end{tikzpicture}
        }
        \caption{$(k,\ell,m,0)$.}
    \end{subfigure}
    \vskip\baselineskip
    \begin{subfigure}{.475\textwidth}
        \centering
        \resizebox{0.8\textwidth}{!}{
        \begin{tikzpicture}[every node/.style={draw, circle, fill, inner sep=0pt}]
         	\tikzset{edge/.style={very thick}}
         	\begin{scope}[xshift=0cm,yshift=0cm]

                \node[] (u1)    at (0,1.5)          {a};
                \node[] (u2)    at (0,-1.5)         {a};
                \node[] (u3)    at (-1.5,0)         {a};
                \node[] (u4)    at (1.5,0)          {a};
                \node[] (u5)    at (1.3747,-0.6)    {a};
                \node[] (u6)    at (1.3747,0.6)     {a};

                \node[] (a)     at (-2.5,0)         {a};
                \node[] (d)     at (2.3747,-0.6)    {a};
                \node[] (c)     at (2.3747,0.6)     {a};
                
                \node[draw = none, fill = none] (ce)    at (2.5,0)   {};
                \node[draw = none, fill = none, left=1mm of a]  (la) {\large $a$};
                \node[draw = none, fill = none, right=1mm of c] (lc) {\large $c$};
                \node[draw = none, fill = none, right=1mm of d] (ld) {\large $d$};
                
                \draw[black, very thick, domain=0:23.59] plot ({1.5*cos(\x)}, {1.5*sin(\x)});
                \draw[black, very thick, domain=23.59:90, dotted] plot ({1.5*cos(\x)}, {1.5*sin(\x)});
                \draw[black, very thick, domain=90:270] plot ({1.5*cos(\x)}, {1.5*sin(\x)});
                \draw[black, very thick, domain=270:336.41, dotted] plot ({1.5*cos(\x)}, {1.5*sin(\x)});
                \draw[black, very thick, domain=336.41:360] plot ({1.5*cos(\x)}, {1.5*sin(\x)});
                
                \draw[edge, dotted] (u1) -- (u2);
                \draw[edge] (u4) -- (ce) node[draw = none, fill = none, above=0.01mm, midway]{\large $f$};
                \draw[edge] (u3) -- (a);
                \draw[edge, dotted] (u5) -- (d);
                \draw[edge, dotted] (u6) -- (c);
            \end{scope}
            \begin{scope}[xshift=0cm,yshift=-3cm]
                \node[] (aa)     at (-1.5,0)            {a};
                \node[] (cc)     at (1.5,0)             {a};
                \node[] (dd)     at (0,-2.598)          {a};
                \draw[edge, red, dashed]                (aa) -- (cc);
                \draw[edge, red, dashed]                (aa) -- (dd);
                \draw[edge, red, dashed]                (cc) -- (dd);
                \draw[edge, bend left, green]           (aa) edge (cc);
                \draw[edge, bend right, green]          (aa) edge (dd);
                \draw[edge, bend left, green]           (cc) edge (dd);

                \node[draw = none, fill = none, left=1mm of aa]  (la) {\large $a$};
                \node[draw = none, fill = none, right=1mm of cc]  (lc) {\large $c$};
                \node[draw = none, fill = none, right=1mm of dd] (ld) {\large $d$};
            \end{scope}
        \end{tikzpicture}
        }
        \caption{$(k,0,m,n)$.}
    \end{subfigure}\hfill
    \begin{subfigure}{.475\textwidth}
        \centering
        \resizebox{0.8\textwidth}{!}{
        \begin{tikzpicture}[every node/.style={draw, circle, fill, inner sep=0pt}]
         	\tikzset{edge/.style={very thick}}
         	\begin{scope}[xshift=0cm,yshift=0cm]
                \draw[black, very thick] (0,0) circle (1.5);
                \node[] (u1)    at (0,1.5)          {a};
                \node[] (u2)    at (0,-1.5)         {a};
                \node[] (u3)    at (-1.5,0)         {a};
                \node[] (u4)    at (1.5,0)          {a};
                \node[] (u5)    at (1.3747,0.6)     {a};
                \node[] (u6)    at (0,0)            {a};
                \node[] (u7)    at (1.3747,-0.6)    {a};
                
                \node[] (a)     at (-2.5,0)         {a};
                \node[] (b)     at (-0.75,0)        {a};
                \node[] (c)     at (2.3747,0.6)     {a};
                \node[] (d)     at (2.3747,-0.6)    {a};
                
                \node[draw = none, fill = none] (ce)    at (2.5,0)   {};
                \node[draw = none, fill = none, left=1mm of a]  (la) {\large $a$};
                \node[draw = none, fill = none, left=1mm of b]  (lb) {\large $b$};
                \node[draw = none, fill = none, right=1mm of c] (lc) {\large $c$};
                \node[draw = none, fill = none, right=1mm of d] (ld) {\large $d$};
                
                \draw[edge] (u1) -- (u2);
                \draw[edge] (u4) -- (ce) node[draw = none, fill = none, above=0.01mm, midway]{\large $f$};
                \draw[edge] (u3) -- (a);
                \draw[edge] (u5) -- (c);
                \draw[edge] (u6) -- (b);
                \draw[edge] (u7) -- (d);
            \end{scope}
            \begin{scope}[xshift=0cm,yshift=-3cm]
                \node[] (a)     at (-1.5,0)             {a};
                \node[] (b)     at (1.5,0)              {a};
                \node[] (c)     at (-1.5,-3)          {a};
                \node[] (d)     at (1.5,-3)           {a};
                \draw[edge, red, dashed]                (a) -- (b);
                \draw[edge, red, dashed]                (a) -- (c);
                \draw[edge, red, dashed]                (a) -- (d);
                \draw[edge, red, dashed]                (b) -- (c);
                \draw[edge, red, dashed]                (b) -- (d);
                \draw[edge, red, dashed]                (c) -- (d);
                \draw[edge, bend left, red, dashed]     (a) edge (b);

                \node[draw = none, fill = none, left=1mm of a]  (la) {\large $a$};
                \node[draw = none, fill = none, right=1mm of b] (lb) {\large $b$};
                \node[draw = none, fill = none, left=1mm of c]  (lc) {\large $c$};
                \node[draw = none, fill = none, right=1mm of d] (ld) {\large $d$};
            \end{scope}
        \end{tikzpicture}
        }
        \caption{$(k,\ell,m,n)$.}
    \end{subfigure}
    
    \caption{Each subfigure shows a pendant level-$2$ blob together with its CAG directly below it.
    On each blob,~$f$ denotes the non-trivial cut-edge.
    Each of the leaves~$a,b,c,d$ can be replaced by a longer chain whilst keeping the same CAG.
    By Theorem~\ref{thm:CAG}, we have that the network contains one of the four pendant blobs if and only if the CAG (which can be obtained from the multisets of distances) is exactly the one in the same subfigure.
    In the CAG, the dashed lines represent the red edges and the solid lines represent the green edges.
    In (c), the green edge~$cd$ in the CAG covers the dotted path between~$c$ and~$d$.}
    \label{fig:CAG}
\end{figure}

\begin{theorem}\label{thm:CAG}
    (See Figure~\ref{fig:CAG}.)
    Let~$N$ be a level-$2$ network on~$X$ with at least two blobs, where no pendant blobs are of the form~$(k,0,0,0)$ and~$(k,\ell,0,0)$ in which all leaves are contained in blobs. For~$k,\ell,m,n\geq1$,~$N$ contains a pendant level-$2$ blob of the form
    \begin{itemize}
        \item $(k,0,m,0)$ if and only if there exist vertices~$a$ and~$c$ which form a blob in the CAG with~$1$ red edge and~$2$ green edges between them.
        \item $(k,\ell,m,0)$ if and only if there exist vertices~$a,b,$ and~$c$ which form a blob in the CAG, where~$a$ and~$b$ are connected by~$2$ red edges and the other two pairs are connected by~$1$ red edge and~$1$ green edge.
        \item $(k,0,m,n)$ if and only if there exist vertices~$a,c,$ and~$d$ which form a blob in the CAG, where every pair of vertices are connected by~$1$ red edge and~$1$ green edge.
        \item $(k,\ell,m,n)$ if and only if there exist vertices~$a,b,c,$ and~$d$ which form a blob in the CAG, where every pair of vertices are connected by~$1$ red edge, and~$a$ and~$b$ are connected by an additional red edge.
    \end{itemize}
\end{theorem}
\begin{proof}
    All other possible pendant level-$2$ blobs are of the form~$(k,0,0,0)$ or of the form~$(k,\ell,0,0)$.
    The CAG of the blob of the form~$(k,0,0,0)$ is the singleton graph; the CAG of the blob of the form~$(k,\ell,0,0)$ is two vertices connected by~$2$ red edges.
    The CAG for either of these two pendant blobs is not the same as any of the CAG for the four pendant blobs that we investigate here.
    Therefore we may distinguish the CAG of the pendant level-$2$ blobs from one another.
    
    Now we consider non-pendant level-$2$ blobs.
    First, if the blob contains no leaves then the CAG of such a blob is empty, so we are done.
    Hence, suppose that some non-pendant level-$2$ blob~$B$ contains some leaves.
    Observe that~$B$ can be obtained by introducing non-trivial cut-edges to one of the six possible level-$2$ pendant blobs.
    
    Suppose first that~$B$ can be obtained by introducing non-trivial cut-edges to a pendant blob of the form~$(k,0,0,0)$.
    Then,~$B$ contains one or more chains on one side of the blob, and the possible CAGs would be a path (or disjoint paths) of red edges that connect adjacent chains, or if it contains a green edge, two vertices that are connected by~$1$ red and~$1$ green edge. 
    However, none of these CAGs correspond to that of the four pendant blobs we consider here.
    
    Now suppose that~$B$ can be obtained by introducing non-trivial cut-edges to a pendant blob of the form~$(k,\ell,0,0)$.
    Then,~$B$ contains one or more chains on two sides of the blob, and at least one non-trivial cut-edge on the third side.
    None of the edges in the CAG of~$B$ will cover an edge of this third side, since all paths between chain endpoints that uses this side will be of length at least~$6$.
    Therefore the only possible CAGs we can get on~$B$ is a cycle or a path (or paths) of red edges, or two vertices connected by~$1$ red and~$1$ green edge.
    
    Suppose now that~$B$ can be obtained by introducing non-trivial cut-edges to one of the four remaining level-$2$ pendant blobs.
    Upon introducing non-trivial cut-edges to the pendant blob, either the number of chains on the blob increases or stays the same.
    
    Suppose first that this number increases.
    In each of the four pendant blobs, we note that every chain is adjacent to every other chain on the blob.
    It is easy to check that adding non-trivial cut-edges to a pendant blob, which results in the increase in the number of chains on the blob, will return a blob in which every chain is not adjacent to every other chain.
    In particular, one side of~$B$ will contain at least two chains.
    \begin{itemize}
    	\item It follows from here that at most three chains are pairwise adjacent in~$B$.
    	Therefore, non-pendant level-$2$ blobs cannot have a CAG that is the same as that of a pendant blob of the form~$(k,\ell,m,n)$.
    	\item So suppose there are three pairwise adjacent chains in~$B$.
    There are two cases.
    Either the three chains are contained in distinct sides of~$B$, or two of the three chains are contained in the same side of~$B$.
    In the former case, we note that there is at least one side of~$B$ that contains two chains.
    Then, one of the three pairwise adjacent chains contained in this side of~$B$ cannot have an edge from it to the two other chains in the CAG, except for the red edge that shows their adjacency.
    In the latter case, there are exactly two chains on one side of~$B$ and one chain on another side of~$B$ that make up the pairwise adjacent chains.
    An edge between the chain vertices in the CAG excluding the red edge, if it exists, must correspond to some path between chain endpoints that uses the edges of the third side of~$B$.
    But since~$B$ is a non-pendant blob, there must be at least one non-trivial cut-edge on this third side of~$B$.
    Therefore any path between chain endpoints that uses this side must be of length at least~$6$.
    This implies that within the CAG, the three pairwise adjacent chains are connected by a single red edge between all pairs of vertices.
    Therefore, non-pendant level-$2$ blobs cannot have a CAG that is the same as that of a pendant blob of the form~$(k,\ell,m,0)$ nor~$(k,0,m,n)$.
    	\item Finally suppose that there are two chains that are adjacent in~$B$.
    For the CAG of~$B$ on these two vertices to be the same as that of~$(k,0,m,0)$, we would need for the two distance-$5$ paths between chain endpoints to pass through (collectively) all three sides of~$B$.
    However, there are at least two chains contained in one side of~$B$, and thus at least one of these two distance-$5$ paths cannot exist.
    Therefore, non-pendant level-$2$ blobs cannot have a CAG that is the same as that of a pendant blob of the form~$(k,0,m,0)$.
    \end{itemize}
    
    On the other hand suppose that the number of chains on the blob stays the same upon adding non-trivial cut-edges to one of the four level-$2$ pendant blobs.
    Note that for these four cases, all edges of the pendant level-$2$ blobs that do not join the neighbours of leaves of the same chain are covered by at least one of the edges in its CAG.
    Upon inserting non-trivial cut-edges to obtain~$B$, we see a change in color of the CAG edge that used to cover the bisected edge (from red to green), or a possible deletion of the edge (if the edge was green to begin with).
    This will clearly result in a blob~$B$ with a CAG that is different to that of the four level-$2$ pendant blobs we consider here.

    Now we consider the CAG of a level-$1$ blob.
    Observe that a CAG of a level-$1$ blob contains a green edge if and only if the level-$1$ blob contains two chains~$a= (a_1,\ldots,a_k)$ and~$b=(b_1,\ldots,b_\ell)$ such that the cycle of the blob is~$up_1p_2\ldots p_kvwq_1q_2\ldots q_\ell u$, where~$p_i$ and~$q_j$ denote the neighbour of~$a_i$ and~$b_j$ for~$i\in[k], j\in[\ell]$, respectively, and~$u,v$ and~$w$ are incident to non-trivial cut-edges.
    This does not result in any of the CAGs of the four pendant level-$2$ blobs.
    Therefore the CAG of a level-$1$ blob cannot be the same as that of a pendant level-$2$ blob of the forms~$(k,0,m,0); (k,\ell,m,0); (k,0,m,n)$.
    Furthermore, at most~$3$ chains can be pairwise adjacent on a level-$1$ blob.
    Hence the CAG of a level-$1$ blob cannot be the same as that of a pendant level-$2$ blob of the form~$(k,\ell,m,n)$.
\end{proof}
    
    Note that pendant level-$2$ blobs of the form~$(k,0,m,0)$ and~$(m,0,k,0)$ will have the same CAG; however, it is straightforward to find the chain that is on the same side of the blob as the non-trivial cut-edge.
    Given the two chains~$a$ and~$c$ in this case,~$c$ is on the same side of the blob as the non-trivial cut-edge if and only if~$|d(x,c_m)| < |d(x,a_k)|$ for all~$x\in X - (a\cup c)$.
    Note also that we may identify the leaf on the chain that is closest to the non-trivial cut-edge, by taking the same leaf~$x$ and letting~$c_m$ be the chain endpoint satisfying~$d_m(c_m,x) < d_m(c_1,x)$.
    A similar argument holds for the pendant level-$2$ blob of the form~$(k,0,m,n)$, in identifying which chain is on the side of the blob without the non-trivial cut-edge.
    We now seek to replace these pendant level-$2$ blobs by a single leaf~$z$ and alter the multisets of distances accordingly.

\begin{lemma}\label{lem:CAG}
    Let~$k,l,m,n\geq1$, and let~$B$ be a pendant level-$2$ blob that is of the form~$(k,0,m,0)$; $(k,\ell,m,0)$; $(k,0,m,n)$; or~$(k,\ell,m,n)$.
    Then we can replace the pendant blob by a leaf~$z$ to obtain a network~$N'$ on~$X' = X\cup \{z\}-(a\cup b\cup c\cup d)$, such that the multisets of distances of~$N'$ contains the elements
    \[d^{N'}(x,y) = 
    \begin{cases}
        d^N(x,y)    & \text{if }x,y\in X'-\{z\} \\
        A-2         & \text{if }y=z,
    \end{cases}
    \]
    where if~$B$ is of the form
    \begin{itemize}
        \item $(k,0,m,0)$, then we uniquely partition~$d(x,c_m)$ into three equal sized sets~$A,B,C$ such that~$A-2 = B-(m+3) = C-(k+m+3)$.
        \item $(k,\ell,m,0)$, then we uniquely partition~$d(x,c_m)$ into three equal sized sets~$A,B,C$ such that~$A-2 = B-(\ell+m+3) = C-(k+m+3)$.
        \item $(k,0,m,n)$, then we uniquely partition~$d(x,c_m)$ into three equal sized sets~$A,B,C$ such that~$A-2 = B-(m+n+3) = C-(k+m+n+3)$.
        \item $(k,\ell,m,n)$, then we uniquely partition~$d(x,c_m)$ into three equal sized sets~$A,B,C$ such that~$A-2 = B-(\ell+m+n+3) = C-(k+m+n+3)$.
    \end{itemize}
\end{lemma}
\begin{proof}
    The proof is analogous to that of Lemma~\ref{lem:L2NetPendantL2BlobChainAdjust}.
\end{proof}

\medskip

We are now ready to prove Theorem~\ref{thm:L2ReconstructibleMultiset}.

\begin{proof}[Proof of Theorem~\ref{thm:L2ReconstructibleMultiset}]
    Let~$N$ be a level-$2$ network on~$X$.
    We show by induction on~$|E(N)|$, the number of edges in~$N$, that level-$2$ networks are reconstructible from their multisets of distances.
    
    If~$N$ contains a cherry or a leaf that is not contained in a blob, then we can identify these structures and reduce them accordingly by Observation~\ref{obs:Cherry} or Lemma~\ref{lem:LeafBlobAdjacencyAdjust}, respectively.
    Then upon reconstructing the reduced network by the induction hypothesis, we can undo the reduction by either replacing the leaf by a cherry or by reattaching the deleted leaf to the rightful cut-edge by Lemma~\ref{lem:LeafBlobAdjacencyAdjust}.
    If~$N$ is a network on a single blob, then we may reconstruct it from its shortest distances by Lemma~\ref{lem:L2 1BSD}, and therefore from its multisets of distances.

    We now assume that~$N$ is a level-$2$ network on at least two blobs such that every leaf is contained within blobs and that there are no pendant subtrees.
    We show that we may identify pendant blobs and replace them by a leaf.
    First note that a chain on a pendant level-$1$ blob can be identified by Lemma~\ref{lem:L2NetPendantL1Blob};
    Lemma~\ref{lem:L2NetPendantL1BlobAdjust} outlines how we can replace the blob by a leaf~$z$ and adjust the multisets of distances accordingly.
    It is easy to reconstruct the blob after reconstructing the reduced network, since we know the chain that is contained in the blob.
    For pendant level-$2$ blobs, recall that they are of the form~$(k,\ell,m,n)$ where~$k,\ell,m,n\geq0$. 
    The following list shows how all possible pendant level-$2$ blobs can be identified with one of the lemmas that we have proven before:
    \begin{itemize}
        \item $(k,0,0,0)$ by Lemmas~\ref{lem:L2NetPendantL2BlobChain} and~\ref{lem:L2NetPendantL2BlobLeaf};
        \item $(k,\ell,0,0)$ by Lemma~\ref{lem:L2NetPendantL2BlobTwoChains}; and
        \item $(k,0,m,0); (k,0,m,n); (k,\ell,m,0);$ and~$(k,\ell,m,n)$ by Theorem~\ref{thm:CAG}.
    \end{itemize}
    Replacing the pendant level-$2$ blobs by a leaf~$z$ and adjusting the multisets of distances accordingly for each case has been outlined in Lemmas~\ref{lem:L2NetPendantL2BlobChainAdjust},~\ref{lem:L2NetPendantL2BlobTwoChainsAdjust}, and~\ref{lem:CAG}.
    It is easy to reconstruct the blob after reconstructing the reduced network, since we know which chains are on the same side of the blob as the non-trivial cut-edge.
    
    Observe that every level-$2$ network has a cherry, exactly one blob, a leaf that is not contained in a blob, or a pendant blob.
    We have now shown that it is possible to identify these structures, to reduce them, and to add these structures back to the reduced network to obtain the original network.
    All these reductions decrease the number of edges of the network.
    Then by the induction hypothesis, we may reconstruct the reduced network from the modified distance matrix -- since we can obtain the original network from the reduced network for each case, this completes the proof.
\end{proof}


\section{Discussion}\label{sec:Discussion}

We have considered the fundamental question of deciding which networks are uniquely reconstructible from 
the pairwise graph-theoretical distances between their leaves.
We showed that level-1 networks are reconstructible from their shortest distances and that level-2 networks are reconstructible from their multisets of distances. 
We have also shown that networks of level higher than~$1$ and level higher than~$2$ are not reconstructible from their shortest distances and multisets of distances in general, respectively (Lemmas~\ref{lem:L2NSD} and~\ref{lem:LkNSM}).

From a practical perspective, having the multisets of distances is not very realistic. 
For example, starting with sequence data, it is not clear how multisets could be
produced in an accurate and efficient manner.
As stated in~\cite{bordewich2016algorithm}, while it may be possible to obtain `...the set of different evolutionary path weights between a given pair of taxa, it seems hard to imagine how one might manage to measure the number of distinct evolutionary paths of a given observed weight.'
Naturally, this points to the idea that perhaps we should investigate other types of distance matrices that are more restrictive when compared to the multisets of distances, that may be relatively easy to obtain from sequence data. 
Therefore in future research, it would be of interest to consider other distance matrices such as tree-average distances~\cite{willson2012tree} and sets of distances~\cite{bordewich2016algorithm}.
In particular, the two level-$2$ networks in Figure~\ref{fig:level2} have the same shortest distance matrices, but different sets of distances  (i.e., 
the underlying sets of their multisets of distances are different).
Therefore the question of whether a level-$2$ network is 
reconstructible from its set of distances remains open.

On a similar note, we wonder if there is some characterization of 
level-$2$ networks that are reconstructible from their shortest distances.
We have already seen instances of this, for example when 
the level-$2$ network contains at most~$3$ leaves (Lemma~\ref{lem:L2<3ShortestRecon}) 
and when the network contains exactly one blob (Lemma~\ref{lem:L2 1BSD}). 
We conjecture that if every side of all blobs have enough incident edges, then 
they should provide enough information for unique reconstructibility.
To motivate this conjecture, 
note that the networks in Figure~\ref{fig:level2} contain a level-$2$ blob of the form~$(2,0,0,0)$.
If every level-$2$ blob has at least two sides with enough cut-edges incident to them so 
that when they become pendant blobs upon reducing the network 
they are not of the form~$(k,0,0,0)$, then is the network reconstructible from its shortest distances?
A similar question can be posed for level-$k$ networks for~$k\ge3$.
Can we characterize level-$k$ networks that are reconstructible from 
their multisets of distances, or possibly from their shortest distances?

On the algorithmic side, the proofs of Theorems~\ref{thm:L1ReconstructibleShortPath} and~\ref{thm:L2ReconstructibleMultiset} outline the steps that can 
be taken to construct networks from distance data.
Indeed, in both the level-$1$ and the level-$2$ cases, we describe 
how one can identify a cherry or a pendant blob, reduce it to a single leaf, and adjust the new distance matrices.
Since all networks contain either a cherry or a pendant blob, we may 
recurse on the reduced instances until there is a tree or a single blob in the 
network, at which point we are done. The important question as to whether this 
algorithm can run in polynomial time remains open.

In practice, even if we are able to find efficient algorithms 
that can uniquely construct level-$1$/level-$2$ networks from their shortest/multisets of distances,
it is important to bear in mind that variations in distances 
arising from real data sets may lead to inconsistencies which cannot be handled by
such algorithms.
One way to deal with such inconsistencies 
would be to consider a slight variant of the problem that we have solved.
As in~\cite{chang2018reconstructing}, we may wish to find an 
unrooted network in which the distance matrix elements correspond 
to the length of some, not necessarily the shortest, path between two taxa.
Though we suspect that the output network will not necessarily be unique, it 
could nonetheless provide a solution that is consistent with the input data
and therefore a useful starting point for making biological deductions.

Finally, a natural extension would be to see if our results generalize to edge-weighted networks.
In addition to considering the network topology, weighted networks take into account edge weights 
which can, for example, represent the amount of genetic divergence that 
has occurred along each edge of the network.
It has been shown that this additional information on the networks can lead to distinguishing two rooted networks on different topologies that display the same set of data (e.g., consider the three distinct rooted level-$1$ networks on three leaves that display the same set of trees)~\cite{pardi2015reconstructible}.
For level-$1$ networks (or for the more general cactus graphs), it was shown recently that while there may exist multiple level-$1$ networks that realize the same shortest distance matrix, there is a unique optimal edge-weighted network whose sum of edge weights is minimal~\cite{hayamizu2020recognizing}.
 It was also 
noted that this is not the case for 
edge weighted, level-$2$ networks by considering an example presented in~\cite{althofer1988optimal}. It could thus
be of interest to ask whether if we consider optimality in terms of the multisets of distances
instead, then is there a unique optimal level-2 network?

\paragraph{Acknowledgements} 

Research funded in part by the Netherlands Organization for Scientific Research (NWO), including Vidi grant 639.072.602, and partly by the 4TU Applied Mathematics Institute. 
Vincent Moulton thanks the Netherlands Organization for Scientific Research (NWO), including Vidi grant 639.072.602, for its support to visit TU Delft.
\medskip

We would like to thank Remie Janssen and Mark Jones for helpful discussions and providing feedback on multiple versions of the paper.






\end{document}